\documentclass[11pt, reqno]{amsart}
\usepackage{mathrsfs}
\usepackage[active]{srcltx}
\usepackage{mathrsfs,amsmath}
\usepackage{mathtools}
\usepackage{longtable}
\usepackage{todonotes}
\usepackage{appendix}
\usepackage[autostyle]{csquotes}
\allowdisplaybreaks

\newcommand{\abs}[1]{\left\lvert #1 \right\rvert}

\newcommand{\RR }{\mathbb{R}}

\UseRawInputEncoding

\usepackage{xspace}
\usepackage{xcolor}
\usepackage[normalem]{ulem}
\usepackage{etoolbox}

\newtoggle{ignoreflag}

\definecolor{bluecite}{HTML}{0875b7}
\usepackage[unicode=true,
bookmarksopen={true},
pdffitwindow=true,
colorlinks=true,
linkcolor=bluecite,
citecolor=bluecite,
urlcolor=bluecite,
hyperfootnotes=false,
pdfstartview={FitH},
pdfpagemode= UseNone]{hyperref}

\newcommand{\R}{{\mathbb R}}

\newtheorem{example}{Example}[section]
\newtheorem{proposition}{Proposition}[section]
\newtheorem{theorem}{Theorem}[section]

\newtheorem{corollary}{Corollary}[section]

\newtheorem{remark}{Remark}[section]
\numberwithin{equation}{section}

\usepackage{bbm}  
\usepackage{dsfont}

\title[Quantitative Stability]{Quantitative stability  for Bakry--\'Emery \\ log-Sobolev and Talagrand inequalities}
\vspace{-0.3cm}\author{Alexandru Krist\'aly$^*$ \& Alexandru P\^irvuceanu}

\address{}

\email{}

\address{
}

\email{}

\address{\textsc{Alexandru Krist\'aly}: Department of Economics, Babe\c s-Bolyai University, str. Teodor Mihali 58-60, 400591, Cluj-Napoca, Romania \& 
	Institute of Applied Mathematics, \'Obuda
	University, B\'ecsi \'ut 96/B, 1034,
	Budapest, Hungary.} 
\email{alex.kristaly@econ.ubbcluj.ro; kristaly.alexandru@uni-obuda.hu}

\address{\textsc{Alexandru P\^irvuceanu}: Department of Mathematics, Babe\c s-Bolyai University, 1 Mihail Kog\u{a}lniceanu, 400084, Cluj-Napoca, Romania.} 
\email{alexandru.pirvuceanu@ubbcluj.ro, pirvuceanualexandrudaniel@gmail.com}

\thanks{Research of A.K. supported by the
	Excellence Researcher Program \'OE-KP-2-2022 of \'Obuda University, Hungary.\\ 
\quad \hspace*{0.25cm} $^*$Corresponding author.
}

\thanks{}

\begin{document}
\vspace{-0.4cm}
\begin{abstract} 
	We establish quantitative $L^
	1$-stability estimates for the Bakry--\'Emery log-Sobolev and Talagrand inequalities with a universal exponent of 1/19 governing the corresponding deficits.
	Our approach relies on a Maurey-type argument combined with stability estimates for the Pr\'ekopa--Leindler inequality.
	In the radial setting, the exponent of both deficits can be improved to $1/2,$ which turns out to be optimal. As an application, we establish an estimate  for the  hypercontractivity deficit  of the Hopf--Lax semigroup in the Bakry--\'Emery setting. 
		In particular, these stability results provide elementary characterizations  for the equality cases in the previously mentioned inequalities.

	
\end{abstract}
\vspace{-0.4cm}
\dedicatory{Dedicated to professor P\'eter T. Nagy on the occasion of his 80th birthday}

\maketitle
\vspace{-1.1cm}
\section{Introduction and Main Results} 

Sharp Sobolev inequalities have a long history, their systematic  study being initiated  in the seminal works of Aubin \cite{Aubin} and Talenti \cite{Talenti}. The sharpness in these inequalities immediately raises the question of characterizing the equality cases; in the case of the classical Sobolev 
 inequality, there is a unique class of extremizers, formed by the family of Talentian 
 bubbles.  An intriguing problem concerning the \textit{stability}  was formulated by  Br\'ezis and Lieb \cite{Brezis-Lieb}:   describe a meaningful estimate for the Sobolev deficit  in terms of some "natural distance" from the manifold of extremizers. Roughly speaking,   how far is a function from the manifold of extremizers whenever its Sobolev deficit is sufficiently small? 

The first  stability result was provided by Bianchi and Egnell \cite{BianchiEgnell} for the classical Sobolev \mbox{inequality}. This pioneering work opened a research direction for stability results not only for Sobolev-type \mbox{functional} inequalities, see, e.g.,\ Cianchi,  Fusco, Maggi, and Pratelli \cite{CFMP}, Deng, Sun, and Wei \cite{DSW}, Figalli and Zhang \cite{FigalliRu}, Dolbeault, Esteban,
Figalli, Frank, and Loss \cite{Dolbeault-et-al}, but also for geometric (isoperimetric, Brunn--Minkowski, Borell--Brascamp--Lieb) inequalities, see, e.g.,\ Figalli, Maggi, and Pratelli \cite{FigalliMaggiPratelli}, Figalli and Jerison \cite{FigalliJerison}, Figalli,  van Hintum, and Tiba \cite{Figalli-vanHintum-Tiba}.   

Besides the classical Sobolev inequality, the \textit{log-Sobolev inequality} has been widely studied as well. Indeed, its various forms appear as indispensable tools to describe nonlinear
phenomena, such as the solution to Poincar\'e's conjecture, see Perelman \cite{Perelman}, hydrodynamic scalings for  interacting particles, see Yau \cite{Yau}, hypercontractivity estimates for Hopf--Lax semigroups, see, e.g.,\
Bobkov, Gentil, and Ledoux \cite{BGL-JMPA}, Gentil \cite{Gentil}, Otto and Villani \cite{Otto-Villani}. 

Although several stability results are available  for log-Sobolev inequalities -- mainly with the  \mbox{Gaussian} measure
-- 
one of our main motivations   is to provide  new quantitative stability estimates in $\mathbb R^n$ in the Bakry--\'Emery setting by adapting a Maurey-type argument via the Pr\'ekopa--Leindler inequality. It turns out that this  approach also successfully applies to establishing stability  for the \textit{Talagrand \mbox{transport} inequality}. These results  can also be applied to estimate the \textit{hypercontractivity deficit}  of the Hopf--Lax semigroup in the Bakry--\'Emery setting. An important byproduct of these stability results is the characterization of the equality in the Talagrand inequality and the hypercontractivity estimate in the general Bakry--\'Emery framework in $\mathbb R^n$, which seem to be missing pieces in the literature.

The following subsections are devoted to the main results of the paper, together with a discussion of their connections to existing works.

\subsection{Bakry--\'Emery  log-Sobolev inequality}
Let ${\rm d}\mu_V=e^{-V}{\rm d}x$  be a probability measure on $\mathbb R^n$ verifying the \textit{Bakry--\'Emery  property}, namely,  $V\in C^2(\mathbb R^n)$ and  $\nabla^2 V -cI_n\geq 0$ for some $c>0.$
The \textit{Bakry--\'Emery  log-Sobolev inequality}, see  \cite{Bakry-Emery}, 
 states that  
\begin{equation}\label{gaussian-logsob-00}
	{\rm Ent}_{\mu_V}(f^2)\le \frac{2}{c}\int_{\RR^n} |{\nabla f}|^2\, {\rm d}\mu_V,\quad  \forall f\in W^{1, 2}(\RR^n, \mu_V),
\end{equation}
where 
$$	{\rm Ent}_{\mu_V}(f^2)=\int_{\mathbb R^n}f^2 \log f^2 {\rm d}\mu_V - \left(\int_{\mathbb R^n}f^2  {\rm d}\mu_V\right) \log\left(\int_{\mathbb R^n}f^2  {\rm d}\mu_V\right)$$
stands for the  entropy of $f^2$ with respect to the measure  ${\rm d}\mu_V,$ while $W^{1, 2}(\RR^n, \mu_V)$ is the family of   functions $f\in L^2(\RR^n, \mu_V)$ with $|\nabla f|\in L^2(\RR^n, \mu_V).$ Alternative proofs of \eqref{gaussian-logsob-00} are provided by   optimal mass transport arguments, see Otto and Villani \cite{Otto-Villani}, Bobkov and Ledoux \cite{BobkovLedoux},  Bobkov,   Gentil,  and Ledoux \cite[Theorem 4.1]{BGL-JMPA}, and  Cordero--Erausquin \cite{Cordero-ARMA}.

Note that the constant $2/c$  is sharp, while  the  characterization of  equality  in \eqref{gaussian-logsob-00} is   known from the work of  Arnold,   Markowich,  Toscani, and   Unterreiter \cite[Theorem 3.11]{CPDE-equality}. 
The most familiar form of  \eqref{gaussian-logsob-00} is the \textit{Gross log-Sobolev inequality} for the Gaussian measure $\gamma:=\mu_{V_G}$, the potential being $$V_G(x)=\frac{1}{2}|x|^2+\frac{n}{2}\log({2\pi}),\quad x\in \mathbb R^n,$$ see Gross \cite{Gross}; the characterization of the equality in this case was first given by  Carlen \cite{Carlen}.

In view of  \eqref{gaussian-logsob-00}, we   introduce the \textit{Bakry--\'Emery log-Sobolev deficit} for $f\in W^{1, 2}(\RR^n, \mu_V)\setminus \{0\}$, namely 
$$\delta^{\sf LSI}_{\mu_V}(f):=\frac{\displaystyle\frac{2}{c}\int_{\RR^n} |{\nabla f}|^2\, {\rm d}\mu_V-{\rm Ent}_{\mu_V}(f^2)}{\displaystyle\int_{\mathbb R^n} f^2{\rm d}\mu_V} \ge 0.$$
It is clear that $\delta^{\sf LSI}_{\mu_V}(\lambda f)=\delta^{\sf LSI}_{\mu_V}(f)$ for every $\lambda>0.$

There is a significant number of stability estimates for \eqref{gaussian-logsob-00} in terms of the above deficit in various metrics, mostly for the Gaussian measure $\gamma=\mu_{V_G}$.
For instance, Indrei and Marcon \cite{Indrei-Marcon} and Bobkov,   Gozlan,   Roberto, and Samson \cite{BGRS} obtained   stability results  with  respect to the  Wasserstein distance $W_2$; similar results are obtained by Indrei and Kim \cite{IndreiKim}  in terms of the $W_1$ and $L^1$  distances. Under the validity of a $(2, 2)$-Poincar\'e inequality, Fathi, Indrei, and Ledoux \cite{FathiIndreiLedoux} established a strict improvement of \eqref{gaussian-logsob-00} for centered probability measures, which in turn yields stability bounds in terms of the $W_2$ and $L^1$ distances. Very recently, Dolbeault, Esteban, Figalli, Frank, and Loss \cite[Corollary 1.2]{Dolbeault-et-al} established a new type of sharp $L^2$-estimate  for the deficit $\delta^{\sf LSI}_{\mu_{V_G}}$. Further stability results for \eqref{gaussian-logsob-00} with the Gaussian measure can be found in the papers by Bez,  Nakamura, and   Tsuji \cite{Bez-etal}, Dolbeault and Toscani \cite{Dolbeault-Toscani}, Ledoux,   Nourdin, and Peccati \cite{LNP}, Suguro \cite{Suguro}, and the references therein. 
  For a recent survey paper, see  Brigati, Dolbeault, and Simonov \cite{Brigati-Dolb-Sim}.


When it comes to stability bounds for \eqref{gaussian-logsob-00} in the case of non-Gaussian probability measures, there are significantly less results available in the literature. Under the Bakry--\'Emery property  with $c=1$, Courtade and Fathi \cite{CourtadeFathi} proved a stability result with respect to the $W_1$ Wasserstein distance for nonnegative functions whose logarithm is a $\lambda-$Lipschitz function, with their constant depending on $\lambda>0$. Gozlan \cite{Gozlan} established a deficit bound with respect to the $W_2$ Wasserstein distance in the case of an unconditional potential $V$, whereas Bolley, Gentil, and Guillin \cite{BolleyGentilGuillin} obtained a dimensional improvement of the log-Sobolev inequality \eqref{gaussian-logsob-00}. 

Motivated by these results, our first goal is to provide a stability estimate for the log-Sobolev inequality \eqref{gaussian-logsob-00} under the validity of the general Bakry--\'Emery  property. For the class of   log-concave functions in $W^{1, 2}(\RR^n, \mu_V)$ whose Bakry--\'Emery log-Sobolev deficit is sufficiently small,  
we have the following:

	%
	%
	%

\begin{theorem} \label{Theorem-LSI} {\rm (Bakry--\'Emery  log-Sobolev inequality)} Let ${\rm d}\mu_V=e^{-V}{\rm d}x$  be a probability measure on $\mathbb R^n$ such that $V\in C^2(\mathbb R^n)$ and  $\nabla^2 V -cI_n\geq 0$ for some $c>0.$ Then there exists a dimension-depending constant $C_n>0$ with the following property$:$ for every positive  log-concave  function $f\in W^{1, 2}(\RR^n, \mu_V)$ with $\delta^{\sf LSI}_{\mu_V}(f)<C_n^{-19},$ 
		there is a point $x_0\in \mathbb R^n$   such that
		\begin{equation}\label{stability-estimate-00}
			\int_{\RR^n}\abs{\frac{f^2(x)}{\alpha}-e^{V(x)-V(x-x_0)}}\, {\rm d}\mu_V(x)\le C_n\cdot (\delta^{\sf LSI}_{\mu_V}(f))^{\frac{1}{19}},
		\end{equation}
		where $\alpha=\displaystyle \int_{\mathbb R^n} f^2{\rm d}\mu_V$.
	\end{theorem} 

The proof of Theorem \ref{Theorem-LSI} is based on a   stability result for the Pr\'ekopa--Leindler inequality due to B\"or\"oczky and De \cite{BoroczkyDe}, combined with a Maurey-type argument  \cite{Maurey}. The latter argument has been successfully applied by   Bobkov and Ledoux \cite{BobkovLedoux} to prove \eqref{gaussian-logsob-00}; see also Bobkov and G\"otze \cite{Bobkov-Gotze}. By means of the improved Pr\'ekopa--Leindler inequality of Bucur and Fragal\`a \cite{BucurFragala}, this argument was used by Feo, Indrei, Posteraro, and Roberto \cite{Feo-Indrei etal} to obtain a stability result for \eqref{gaussian-logsob-00} in the case of the Gaussian measure. Under an additional growth condition for the log-concave function $f$,  Balogh and Krist\'aly \cite[Theorem 5.3]{Balogh-Kristaly-preprint}  obtained relation \eqref{stability-estimate-00} very recently in the case of the Gaussian measure.

We emphasize that the only condition on the probability measure $\nu=\frac{f^2}{\alpha} \mu_V$ is the log-concavity of its density.  This is strongly related to the problem raised by Indrei and Kim \cite{IndreiKim} 
 concerning the largest class of probability densities for which an $L^1$-stability   holds for the (Gross) log-Sobolev inequality \eqref{gaussian-logsob-00}. Note that in \cite{IndreiKim} the   class of centered probability measures with a  uniform  second momentum bound is shown to work. In addition, our estimate \eqref{stability-estimate-00} is valid not only for the 
Gaussian measure, but also for measures $\mu_V$ with $V$ verifying the  
 Bakry--\'Emery property.
 
 
The use of  the  Pr\'ekopa--Leindler stability from B\"or\"oczky and De \cite{BoroczkyDe} requires the log-concavity of $f$. Although there are more recent stability results for the Pr\'ekopa--Leindler   inequality without such log-concavity restrictions, see Figalli, van Hintum, and Tiba \cite{Figalli-vanHintum-Tiba},  certain technicalities prevent the aforementioned Maurey-type approach to provide meaningful deficit bounds. 

At a first glance, relation  \eqref{stability-estimate-00} implies that in the proximity of the log-concave function $f^2$ there is a not necessarily log-concave function $x\mapsto e^{V(x)-V(x-x_0)}$ for some $x_0\in \mathbb R^n$. 
However,  relation \eqref{stability-estimate-00} yields a   \textit{rigidity} for $V$ when equality holds in \eqref{gaussian-logsob-00}: one has
$\delta^{\sf LSI}_{\mu_V}(f)=0$ for a positive log-concave   $f\in W^{1, 2}(\RR^n, \mu_V)$   if and only if $f(x)=a\cdot e^{\langle x_0,x\rangle}$ for some $a>0$  and $$x_0\in {\rm Ker}(\nabla^2 V -cI_n):=\bigcap_{x\in \mathbb R^n}{\rm ker}(\nabla^2 V(x) -cI_n),$$
where ${\rm ker}(M)$ is the null space (or the kernel) of the $n\times n$ matrix $M$.
%
%
%
In fact, the latter algebraic property is equivalent to the fact  that for some $C_0\in \mathbb R$, one has $V(x)-V(x-x_0 )=c\langle x_0,x\rangle +C_0$ for every $x\in \mathbb R^n,$ cf.\   Proposition \ref{V-properties}, which clearly explains the profile of the extremizer $f$ in \eqref{gaussian-logsob-00} coming from the expression $x\mapsto e^{V(x)-V(x-x_0)}$. 
This algebraic formulation is an alternative characterization of  the equality case in \eqref{gaussian-logsob-00} expressed  analytically  by Arnold,   Markowich,  Toscani, and   Unterreiter \cite{CPDE-equality};   in    Remark \ref{remark-equal} we compare the two approaches.  The above characterization is also valid without the log-concavity of $f$, which we   obtain by the famous HWI inequality of Otto and Villani \cite{Otto-Villani} and  a stability of 
Talagrand's  inequality, see Corollary \ref{corollary-LSI=}.

The exponent $1/19$ of the deficit $\delta^{\sf LSI}_{\mu_V}(f)$ in   \eqref{stability-estimate-00} is not expected to be optimal, as Indrei and Kim \cite[Theorem 1.1]{IndreiKim} have already stated an $L^1$-deficit estimate with the exponent $1/4$  for centered measures having a uniformly bounded second momentum. In fact, Kim \cite[Remark 1.2]{Kim} formulated -- for the Gaussian measure  $\gamma=\mu_{V_G}$ -- the conjecture that the optimal exponent in the $L^1$-deficit estimate should be between $1/4$ and $1.$ In the setting of  \cite{IndreiKim} (i.e., for centered measures having a uniformly bounded second momentum),   Indrei \cite{IndreiJDE} stated the optimality of the exponent 1/2  with respect to the $W_1$ Wasserstein metric.  The following result confirms Kim's conjecture  in the   \textit{radial} case, by proving the  optimality of the exponent 1/2 in the $L^1$-deficit estimate:




\begin{theorem} \label{Theorem-LSI-radial} {\rm (Radial Bakry--\'Emery  log-Sobolev inequality)} Let ${\rm d}\mu_V=e^{-V}{\rm d}x$  be a probability measure on $\mathbb R^n$ such that $V\in C^2(\mathbb R^n)$ is radial and  $\nabla^2 V -cI_n\geq 0$ for some $c>0.$ The following statements hold$:$ 
	
	\begin{enumerate}
		\item[(i)] {\rm (Stability)} There exists a dimension-depending constant $C_n>0$ with the following property$:$ for every positive log-concave and radial function $f\in W^{1, 2}(\RR^n, \mu_V)$,  one has 
		\begin{equation}\label{stability-estimate-00-radial}
			\int_{\RR^n}\abs{\frac{f^2(x)}{\alpha}-1}\, {\rm d}\mu_V(x)\le C_n\cdot (\delta^{\sf LSI}_{\mu_V}(f))^{\frac{1}{2}},
		\end{equation}
		where $\alpha=\displaystyle \int_{\mathbb R^n} f^2{\rm d}\mu_V$.
		\item[(ii)] {\rm (Optimality)} The exponent $\frac{1}{2}$ of $\delta^{\sf LSI}_{\mu_V}(f)$ in \eqref{stability-estimate-00-radial} is optimal, in the sense that it cannot be replaced by a larger absolute constant for the family of probability measures  ${\rm d}\mu_V=e^{-V}{\rm d}x$ with $V$ satisfying the above Bakry--\'Emery property.  
	\end{enumerate}  
\end{theorem} 


The proof of the first part from Theorem \ref{Theorem-LSI-radial} is based on a recent result of Figalli and Ramos \cite{FigalliRamos}, which gives the expected constant  $1/2$. Furthermore, a direct construction involving a carefully chosen class of test-functions proves the sharpness of the exponent $1/2$.

\subsection{Talagrand transport inequality}
As before, let ${\rm d}\mu_V=e^{-V}{\rm d}x$ be a probability measure on $\mathbb R^n$ with $V$ verifying the Bakry--\'Emery  property from Theorem \ref{Theorem-LSI}. If $\nu$ is a probability measure on $\RR^n$ that is absolutely continuous with respect to $\mu_V $ (i.e., $\nu\ll \mu_V$), having density 
$\frac{{\rm d}\nu}{{\rm d}\mu_V},$
we define its \textit{relative entropy} (also called \textit{informational divergence}) by 
$$D(\nu || \mu_V)=\int_{\RR^n} \frac{{\rm d}\nu}{{\rm d}\mu_V}\log \frac{{\rm d}\nu}{{\rm d}\mu_V} \, {\rm d}\mu_V.$$
The  \textit{Wasserstein distance} between the measures $\nu$ and $ \mu_V$ is given by
$$W_2 (\nu, \mu_V)=\inf\left(\int_{\mathbb R^n}\int_{\mathbb R^n}|x-y|^2{\rm d}\pi(x,y)\right)^{1/2},$$
where the infimum is taken over all probability measures $\pi$ on $\mathbb R^n\times \mathbb R^n$ with
marginal distributions $\nu$ and $\mu_V$.  
The connection between these notions is given by the
\textit{Talagrand transport  inequality}, namely 
\begin{equation}\label{Talagrand-inequality}
	W_2^2(\nu, \mu_V)\le \frac{2}{c}D(\nu||\mu_V),
\end{equation}
see, e.g.,\ Bobkov and Ledoux \cite{BobkovLedoux}, Otto and Villani \cite{Otto-Villani}. Based on \eqref{Talagrand-inequality}, we introduce the \textit{Talagrand deficit} of the probability measure $\nu\ll \mu_V$ with respect to $\mu_V$, defined by 
$$\delta^{\sf Tal}_{\mu_V}(\nu):=D(\nu||\mu_V)-\frac{c}{2}W_2^2(\nu, \mu_V)\ge 0.$$

 In contrast to the stability of log-Sobolev inequalities,  there are much fewer results concerning the stability of the Talagrand transport inequality. As far as we know, the first stability result for the Talagrand inequality is due to Gozlan, Roberto, and Samson \cite{GRS}.  
 Fathi, Indrei, and Ledoux \cite{FathiIndreiLedoux} initiated  a systematic study of this problem in the Gaussian case. For probability measures absolutely continuous with respect to the standard Gaussian measure that satisfy a Poincar\'e inequality, they obtained deficit bounds in terms of the $W_{1, 1}$ Wasserstein distance. Shortly after \cite{FathiIndreiLedoux}, Fathi \cite{Fathi} \mbox{established} an improved inequality for centered probability measures, thus proving an $L^1$-stability for the \mbox{Talagrand}  inequality in the case of the standard Gaussian measure; however, his exponent depends on the dimension and behaves poorly in large dimensions. In Fathi's setting,  Mikulincer \cite{Mikulincer} provided a dimension-free improvement of Talagrand's Gaussian  transport-entropy inequality. For further results on the stability of Talagrand's inequality, see Bez,  Nakamura,  and   Tsuji \cite{Bez-etal},  Cordero-Erausquin  \cite{Cordero-Canadian}, Kolesnikov and Werner \cite{Koleshnikov-Werner}, Feo, Indrei, Posteraro, and Roberto \cite{Feo-Indrei etal}, and    the references therein. 
 
 Our stability result for the Talagrand transport inequality is in the same spirit as Theorem \ref{Theorem-LSI}:

\begin{theorem} \label{Theorem-Talagrand} {\rm (Talagrand transport inequality)} Let ${\rm d}\mu_V=e^{-V}{\rm d}x$  be a probability measure on $\mathbb R^n$ such that $V\in C^2(\mathbb R^n)$ and  $\nabla^2 V -cI_n\geq 0$ for some $c>0.$ The following statements hold$:$ 
	
	\begin{enumerate}
		\item[(i)] {\rm (Stability)} There exists a dimension-depending constant $C_n>0$ with the following property$:$ for every probability measure $\nu\ll \mu_V$ on $\mathbb R^n$ with $\delta^{\sf Tal}_{\mu_V}(\nu)<C_n^{-19},$ there is $x_0\in \mathbb R^n$ such that  
		\begin{equation}\label{deficit-estimate-talagrand}
			\int_{\RR^n}\abs{1-\frac{e^{Lg(x-x_0)+V(x)-V(x-x_0)}}{\int_{\RR^n}e^{Lg}\, {\rm d}\mu_V}}\, {\rm d}\mu_V\le C_n\cdot \left(\delta^{\sf Tal}_{\mu_V}(\nu)\right)^{\frac{1}{19}},
		\end{equation}
		where the function $g:\mathbb R^n\to \mathbb R$ is such that $T(x):=x+\frac{1}{c}\nabla g(x)$ represents the Brenier map pushing forward the measure $\mu_V$ into $\nu $, while $Lg$ is the inf-convolution of $g$ with the  cost $c(x,y)=\frac{c}{2}|x-y|^2$.  
		
		\item[(ii)] {\rm (Equality)} Equality holds in \eqref{Talagrand-inequality} for some probability measure $\nu \ll \mu_V$ on $\mathbb R^n$  if and only if   ${\rm d}\nu=e^{-V(\cdot +x_0)}{\rm d}x$, where  $x_0\in  {\rm Ker}(\nabla^2 V -cI_n)$.

	\end{enumerate}  
\end{theorem}

The proof of Theorem \ref{Theorem-Talagrand}  relies again on the stability result of B\"or\"oczky and De \cite{BoroczkyDe} for the Pr\'ekopa--Leindler inequality combined with a Maurey-type argument \`a la  Bobkov and Ledoux \cite{BobkovLedoux}. In addition, stability in  Talagrand's transport inequality requires -- by its very definition --   subtle arguments  from  optimal mass transport theory and  Kantorovich duality theory, as both of them appear in the statement of Theorem \ref{Theorem-Talagrand}/(i). 

The equality case in Talagrand's transport inequality with the Gaussian reference measure has been established, see Cordero-Erausquin \cite{Cordero-ARMA}, Fathi \cite{Fathi}, Otto and Villani \cite{Otto-Villani}, and only  occurs for measures that are translates of the Gaussian measure.  In the case of an arbitrary reference probability measure $\mu_V$ with $V$ satisfying the Bakry--\'Emery property,   the full characterization of the equality case seems to be new, and it is based on \eqref{deficit-estimate-talagrand} and the Donsker--Varadhan variational formula \cite{Donsker-Varadhan}.
Furthermore, we note that within this general framework, the equality does \textit{not} imply that $\mu_V$ is Gaussian, see Bolley, Gentil, and Guillin \cite[Remark 2.7]{BolleyGentilGuillin}, but rather that $\mu_V$ admits a Gaussian factor encapsulated in  the fact that $V$ is quadratic in the directions spanned by the  vectors of ${\rm Ker}(\nabla^2 V -cI_n)$; see   Remark \ref{remark-equal}, as well as Examples \ref{example-1} \& \ref{example-2}. 

 From a conceptual standpoint, the above stability results -- via the  Pr\'ekopa--Leindler stability and  Maurey's argument --  provide a  novel, elementary approach to  characterizing the equality cases  in a broad class of functional inequalities,  
 including log-Sobolev and Talagrand  inequalities.  
As emphasized by Cordero-Erausquin \cite[p.~263]{Cordero-ARMA} and Otto and Villani \cite[p.~390]{Otto-Villani}, the main obstruction to identifying equality cases via purely optimal transport theory  lies in the lack of regularity of the Brenier map and the corresponding Monge--Amp\`ere equation.

%
%
%
%
%
%
%

In the radial case, similar to Theorem \ref{Theorem-LSI-radial}, we have the following optimal stability:

\begin{theorem} \label{Theorem-Talagrand-radial} {\rm (Radial Talagrand transport inequality)} Let ${\rm d}\mu_V=e^{-V}{\rm d}x$  be a probability measure on $\mathbb R^n$ such that $V\in C^2(\mathbb R^n)$ is radial and  $\nabla^2 V -cI_n\geq 0$ for some $c>0.$ The following statements hold$:$ 
	
	\begin{enumerate}
		\item[(i)] {\rm (Stability)} There exists a dimension-depending constant $C_n>0$ with the following property$:$ for every probability measure $\nu\ll \mu_V$ on $\mathbb R^n$ with radial density,  one has 
		\begin{equation}\label{deficit-estimate-talagrand-radial}
			\int_{\RR^n}\abs{1-\frac{e^{Lg(x)}}{\int_{\RR^n}e^{Lg}\, {\rm d}\mu_V}}\, {\rm d}\mu_V\le C_n\cdot \left(\delta^{\sf Tal}_{\mu_V}(\nu)\right)^{\frac{1}{2}},
		\end{equation}
		where $g:\mathbb R^n\to \mathbb R$ is the function appearing in Theorem \ref{Theorem-Talagrand}/$({\rm i})$.

		\item[(ii)] {\rm (Optimality)} The exponent $\frac{1}{2}$ of $\delta^{\sf Tal}_{\mu_V}(\nu)$ in \eqref{deficit-estimate-talagrand-radial} is optimal, in the sense that it cannot be replaced by a larger absolute constant for the family of probability measures  ${\rm d}\mu_V=e^{-V}{\rm d}x$ with $V$ satisfying the above Bakry--\'Emery property. 
	\end{enumerate}  
\end{theorem}

\subsection{Hypercontractivity estimates for Hopf--Lax semigroups}
 As an application of the above results, we state estimates  for the  hypercontractivity deficit  of the Hopf--Lax semigroup under the Bakry--\'Emery condition. To do this, let ${\rm d}\mu_V=e^{-V}{\rm d}x$  be a probability measure on $\mathbb R^n$ such that $V\in C^2(\mathbb R^n)$ and  $\nabla^2 V -cI_n\geq 0$ for some $c>0.$
For  $t>0$ and  $u:\mathbb R^n\to \mathbb R$, we consider the Hopf--Lax semigroup given by the formula
\begin{equation}\label{inf-convolution-0}
	{\bf Q}_{t}u(x):=\inf_{y\in \mathbb R^n}\left\{u(y)+\frac{|x-y|^2}{2t}\right\}, \ x\in \mathbb R^n;
\end{equation}
by convention, ${\bf Q}_{0}u=u.$ Given a non-decreasing function $q:[0,\infty)\to (0,\infty)$, the hypercontractivity estimate for the Hopf--Lax semigroup -- in the sense of Gross \cite{Gross} -- reads as 
\begin{equation}\label{hypercontractivity-estimate-Gauss-1}
	\|e^{{\bf Q}_tu}\|_{L^{q(t)}(\mathbb R^n,{\rm d}\mu_V)}\leq \|e^{u}\|_{L^{q(0)}(\mathbb R^n,{\rm d}\mu_V)},\ \  t\geq 0.
\end{equation}
Bobkov, Gentil, and Ledoux \cite[Theorem 2.1]{BGL-JMPA}  proved the validity of \eqref{hypercontractivity-estimate-Gauss-1}  for $q(t)=a+ct$ (with $a>0$ arbitrary) and 
any smooth function $u:\mathbb R^n\to \mathbb R$. 

Note that in some cases the estimate \eqref{hypercontractivity-estimate-Gauss-1} becomes trivial. Indeed, for instance, if $u(x)=-|x|^\alpha$ for some $\alpha>2$, it turns out that ${\bf Q}_tu(x)=-\infty$ for every $x\in \mathbb R^n$ and $t>0$, thus the  left-hand side of \eqref{hypercontractivity-estimate-Gauss-1} is zero. In order to overcome such situations, we introduce the class of smooth functions $u:\mathbb R^n\to \mathbb R$ verifying the following property (see also Balogh and Krist\'aly \cite{Balogh-Kristaly-JEMS}): there exist constants $C_1,C_2>0$ and $\theta\in (0,2)$ such that 
\begin{equation}\label{growth-u}
	|u(x)|\leq C_1+C_2|x|^\theta,\ \ x\in \mathbb R^n.	
\end{equation}
Combining condition \eqref{growth-u}  with the super-quadratic feature of the potential $V$ -- which comes from the Bakry--\'Emery property $\nabla^2 V -cI_n\geq 0$ --   a simple argument shows that the terms in \eqref{hypercontractivity-estimate-Gauss-1} are well-defined and non-trivial. 
Having this class of functions,  
we consider the \textit{hypercontractivity deficit} given by 
$$\delta^{\sf HC}_t(u)=\log \frac{\|e^{u}\|_{L^{q(0)}(\mathbb R^n,{\rm d}\mu_V)}}{\|e^{{\bf Q}_tu}\|_{L^{q(t)}(\mathbb R^n,{\rm d}\mu_V)}}\geq 0.$$
We prove the following relationship between the Bakry--\'Emery hypercontractivity and log-Sobolev deficits, which lies at the core of describing the equality case in \eqref{hypercontractivity-estimate-Gauss-1}: 

\begin{theorem}\label{prop-HC-log-Sob}
	Let ${\rm d}\mu_V=e^{-V}{\rm d}x$  be a probability measure on $\mathbb R^n$ such that $V\in C^2(\mathbb R^n)$ and  $\nabla^2 V -cI_n\geq 0$ for some $c>0,$ and let $q:[0,\infty)\to (0,\infty)$ be a non-decreasing $ C^1$-function such that $q'\leq c$. The following statements hold$:$ 
	\begin{itemize}
		\item[(i)]
		For every $t>0$ and every smooth function $u:\mathbb R^n\to \mathbb R$ verifying the growth condition \eqref{growth-u}, one has 
		\begin{equation}\label{HC-LSI-deficits}
			\delta^{\sf HC}_t(u)\geq c\int_0^t\frac{1}{q^2(\tau)}\delta^{\sf LSI}_{\mu_V}(e^{\frac{q(\tau)}{2}{\bf Q}_\tau u}) {\rm d}\tau.
		\end{equation}
		\item [(ii)]	Equality holds in \eqref{hypercontractivity-estimate-Gauss-1} for some $t >0$ and function $u:\mathbb R^n\to \mathbb R$ verifying the growth condition \eqref{growth-u} if and only if there exist  $x_0\in  {\rm Ker}(\nabla^2 V -cI_n)$ and  a number $c_0\in \mathbb R$  such that $u(x)=\langle x_0,x\rangle + c_0,$ $x\in \mathbb R^n $, and for some $a>0$, $q(\tau)=a+c\tau$ for every $\tau\in [0,t )$    whenever $x_0\neq 0.$
	\end{itemize}

\end{theorem}

The paper is organized as follows. In Section \ref{section-2} we recall/state some stability results concerning the Pr\'ekopa--Leindler inequality and basic properties of potentials $V:\mathbb R^n\to \mathbb R$ satisfying the Bakry--\'Emery property. In Section  \ref{section-3} we deal with the Bakry--\'Emery  log-Sobolev inequality by proving  Theorems \ref{Theorem-LSI} and \ref{Theorem-LSI-radial}. In a similar manner, in Section \ref{section-4} we treat the Talagrand transport inequality, providing the  proofs of Theorems \ref{Theorem-Talagrand} and \ref{Theorem-Talagrand-radial}. Finally, in Section \ref{section-5} we prove Theorem \ref{prop-HC-log-Sob}.

\section{Preliminaries}\label{section-2}

\subsection{Stability in the Pr\'ekopa--Leindler inequality}

Recall the following stability result for the Pr\'ekopa--Leindler inequality. 

\begin{theorem}[B\"or\"oczky and De  \cite{BoroczkyDe}] \label{Theorem_Boroczky-De}
	For some absolute constant $c>1$, if $\tau\in (0, \frac{1}{2}]$, $\tau\le \lambda\le 1-\lambda$, $w, u, v:\RR^n\to [0, \infty)$ are integrable functions such that $$w((1-\lambda)x+\lambda y)\ge u(x)^{1-\lambda} v(y)^{\lambda}, \qquad \forall x, y\in \RR^n,$$ $w$ is log-concave, and for some $\varepsilon \in (0, 1]$, $$\int_{\RR^n} w(x)\, {\rm d}x\le (1+\varepsilon)\left(\int_{\RR^n}u(x)\, {\rm d}x\right)^{1-\lambda}\left(\int_{\RR^n} v(x)\, {\rm d}x\right)^{\lambda},$$  then there exists $x_0\in \RR^n$ such that, setting $a=\int_{\RR^n}v/\int_{\RR^n} u$, we have \begin{equation}\label{BDe1}
		\int_{\RR^n}\abs{u(x)-a^{-\lambda}w(x-\lambda x_0)}\, {\rm d}x\le c^n n^n \left(\frac{\varepsilon}{\tau}\right)^{\frac{1}{19}}\int_{\RR^n}u(x) {\rm d}x,    
	\end{equation}
 \begin{equation}\label{BDe2}
		\int_{\RR^n} \abs{v(x)-a^{1-\lambda}w(x+(1-\lambda)x_0)}\, {\rm d}x \le c^n n^n \left(\frac{\varepsilon}{\tau}\right)^{\frac{1}{19}}\int_{\RR^n}v(x) {\rm d}x.
	\end{equation}
\end{theorem} 

The main ingredient in obtaining our stability results in both the Bakry--\'Emery log-Sobolev  and Talagrand inequalities will be the following result; see Machado and Ramos \cite[Lemma 2.7]{MR} for the $\lambda=1/2$ case. 

\begin{proposition}\label{Stability_nonsharp}
	Let $u, v, w:\RR^n\to [0, \infty)$ be integrable functions satisfying 
	$$w((1-\lambda)x+\lambda y)\ge u(x)^{1-\lambda} v(y)^{\lambda}, \qquad \forall x, y\in \RR^n,$$ for some $\lambda\in (0, 1)$. Suppose that $w$ is log-concave. If there is some $\varepsilon>0$ such that $$\int_{\RR^n} w(x) \, {\rm d}x\le (1+\varepsilon)\left(\int_{\RR^n} u(x) \, {\rm d}x\right)^{1-\lambda} \left(\int_{\RR^n} v(x) \, {\rm d}x\right)^{\lambda},$$ then there exist $x_0\in \RR^n$ and a dimensional constant $C_n>0$ such that $$\int_{\RR^n} \abs{\frac{u(x)}{\int_{\RR^n}u}-\frac{v(x-x_0)}{\int_{\RR^n} v}}\, {\rm d}x\le C_n \left(\frac{\varepsilon}{\tau}\right)^{\frac{1}{19}},$$ where $\tau=\min(\lambda, 1-\lambda)$. 
\end{proposition}
\begin{proof}
	Note that $\tau\in (0, \frac{1}{2}]$ since $\lambda\in (0, 1)$. Moreover, it is clear that $\tau\le \lambda \le 1-\tau$. Applying Theorem \ref{Theorem_Boroczky-De}, there exists $x_0\in \RR^n$ such that \eqref{BDe1} and \eqref{BDe2} hold.

	Using the triangle inequality and making a change of variables, we have 
	\begin{align*}
		\int_{\RR^n} \abs{a u(x)-v(x-x_0)}\, {\rm d}x&\le \int_{\RR^n} \abs{a u(x)-a^{1-\lambda}w(x-\lambda x_0)}\, {\rm d}x + \int_{\RR^n} \abs{a^{1-\lambda}w(x-\lambda x_0)-v(x-x_0)}\, {\rm d}x\\
		&=a\int_{\RR^n} \abs{u(x)-a^{-\lambda}w(x-\lambda x_0)}\, {\rm d}x+\int_{\RR^n}\abs{v(x)-a^{1-\lambda}w(x+(1-\lambda)x_0)}\, {\rm d}x.
	\end{align*}
	Now \eqref{BDe1} and \eqref{BDe2} yield $$\int_{\RR^n} \abs{au(x)-v(x-x_0)}\, {\rm d}x\le c^n n^n \left(\frac{\varepsilon}{\tau}\right)^{\frac{1}{19}}\left(a \int_{\RR^n} u(x)\, {\rm d}x+\int_{\RR^n} v(x)\, {\rm d}x\right)=2c^n n^n \left(\frac{\varepsilon}{\tau}\right)^{\frac{1}{19}}\left(\int_{\RR^n} v(x)\, {\rm d}x\right).$$
	Dividing both sides of the above inequality by $\displaystyle \int_{\RR^n} v(x) \, {\rm d}x$ and denoting $C_n:=2c^n n^n>0$, the desired inequality follows. 
\end{proof}

Now recall that in the radial case we have the following improved stability result due to Figalli and Ramos \cite{FigalliRamos}.

\begin{theorem}[Figalli and Ramos \cite{FigalliRamos}] \label{Theorem_Figalli_Ramos}
	For every $n \in \mathbb{N}$, there exists a dimensional  constant $ C(n) >0$ with the following property.
	Let $0<\lambda<1$, and let $u, v, w: \mathbb{R}^n \rightarrow \mathbb{R}_{+}$ be radially symmetric functions such that  either $w$ is log-concave, or both $u$ and $v$ are log-concave. Furthermore, suppose that 
	$$w((1-\lambda)x+\lambda y)\ge u(x)^{1-\lambda} v(y)^{\lambda}, \qquad \forall x, y\in \RR^n,$$ and  for some $\varepsilon\in (0,1],$ we have $$\int_{\RR^n} w(x)\, {\rm d}x=(1+\varepsilon)\left(\int_{\RR^n} u(x)\, {\rm d}x\right)^{1-\lambda} \left(\int_{\RR^n} v(x)\, {\rm d}x\right)^{\lambda}.$$ 
	Then there exist a radially symmetric log-concave function $h: \mathbb{R}^n \rightarrow \mathbb{R}_{+}$ such that
	$$
	\begin{aligned}
		\int_{\R^n}| u(x) - a^{-\lambda}  h\left(x\right)| \mathrm{d} x & \leq C(n)\left(\frac{\varepsilon}{\tau}\right)^{1 / 2} \int_{\mathbb{R}} u(x) \mathrm{d} x, \\
		\int_{\R^n}|v(x)- a^{1-\lambda}h\left(x\right)| \mathrm{d} x & \leq C(n)\left(\frac{\varepsilon}{\tau}\right)^{1 / 2} \int_{\mathbb{R}} v(x)  \mathrm{d} x, \\
		\int_{\mathbb{R}}|w(x)-h(x)|  \mathrm{d} x & \leq C(n)\left(\frac{\varepsilon}{\tau}\right)^{1 / 2} \int_{\R^n} w(x)  \mathrm{d} x,
	\end{aligned}
	$$
	where $a = \int_{\R^n} v/\int_{\R^n} u   $ and $\tau= \min(\lambda, 1-\lambda)$. 
\end{theorem}
Using Theorem \ref{Theorem_Figalli_Ramos}, the same strategy we used to deduce Proposition \ref{Stability_nonsharp} from Theorem \ref{Theorem_Boroczky-De} yields the following result. 
\begin{corollary}\label{Stability_sharp-corollary}
	Under the same assumptions as in Theorem \ref{Theorem_Figalli_Ramos}, there exists a dimension-depending constant $C_n>0$ such that $$\int_{\RR^n} \abs{\frac{u(x)}{\int_{\RR^n} u}-\frac{v(x)}{\int_{\RR^n}v}}\, {\rm d}x\le C_n \left(\frac{\varepsilon}{\tau}\right)^{\frac{1}{2}}.$$
\end{corollary}

\subsection{Bakry--\'Emery potential} 
Let $V\in C^2(\mathbb R^n)$ be such that   $\nabla^2 V -cI_n\geq 0$ for some $c>0.$ It is clear that this property is equivalent to \begin{equation}\label{V-strict-convex}
	 sV(x)+(1-s)V(y)-V(sx+(1-s)y)\geq \frac{c{s(1-s)}}{2}|x-y|^2,\quad \forall x,y\in \mathbb R^n,s\in [0,1].
\end{equation}
The following properties of $V$ will be useful in our proofs: 

\begin{proposition}\label{V-properties}
	Let $V\in C^2(\mathbb R^n)$ be such that   $\nabla^2 V -cI_n\geq 0$ for some $c>0,$ and let $x_0\in \mathbb R^n.$ Then the following statements are equivalent$:$
	\begin{itemize}
		\item[(i)] $x_0\in  {\rm Ker}(\nabla^2 V -cI_n);$
		\item[(ii)] $V(x+x_0)-V(x )=c\langle x_0,x\rangle +V(x_0)-V(0)$ for every $x\in \mathbb R^n;$
		\item[(iii)] $V(x+x_0)-V(x )=\langle \nabla V(x),x_0\rangle +\frac{c}{2}|x_0|^2$ for every $x\in \mathbb R^n.$
	\end{itemize}
\end{proposition}
\begin{proof} If $x_0=0$, we have nothing to prove; thus, let $x_0\neq 0.$
	On the one hand, for every $x\in \mathbb R^n$ one has that 
	\begin{equation}\label{V-prop-1}
		\nabla V(x+x_0)-\nabla V(x)=\int_0^1 \nabla^2 V(x+tx_0)\cdot x_0\, {\rm d}t= \int_0^1 \left(\nabla^2 V(x+tx_0)-cI_n\right)\cdot x_0\, {\rm d}t +cx_0. 
	\end{equation}

	(i)$\implies$(ii) Since $x_0\in  {\rm Ker}(\nabla^2 V -cI_n)$,  by \eqref{V-prop-1} one has
	$	\nabla V(x+x_0)-\nabla V(x)=cx_0$ for every $x\in \mathbb R^n$; thus 
	$V(x+x_0)- V(x)=c\langle x_0,x\rangle  + c_0$ for some $c_0\in \mathbb R^n$. It is clear that $c_0=V(x_0)-V(0)$.
	
	(ii)$\implies$(i) By \eqref{V-prop-1}, it follows that for every $x\in \mathbb R^n$ one has
	$$\int_0^1 \langle x_0,\left(\nabla^2 V(x+tx_0)-cI_n\right)\cdot x_0\rangle\, {\rm d}t=0.$$
	Since	$\nabla^2 V -cI_n\geq 0$, the latter relation implies that $\langle x_0,\left(\nabla^2 V(x+tx_0)-cI_n\right)\cdot x_0\rangle=0$ for every $t\in (0,1)$ and $x\in \mathbb R^n$; in fact, by an algebraic argument\footnote{If $A$  is a symmetric positive semi-definite $n\times n$ matrix and $x_0\in \mathbb R^n$ with $\langle x_0,A\cdot x_0\rangle=0$, then $A\cdot x_0 =0.$}, one has   that $\left(\nabla^2 V(x+tx_0)-cI_n\right)\cdot x_0=0$, that is,  $x_0$ is in the null space of   $\nabla^2 V(x)-cI_n$ for every $x\in \mathbb R^n,$ i.e., $x_0\in  {\rm Ker}(\nabla^2 V -cI_n)$.
	
	On the other hand, by a Taylor expansion, one has for every $x\in \mathbb R^n$ that
	\begin{align*}\label{appendix-V-expans}
		\nonumber	V(x+x_0)-V(x)&=\langle \nabla V(x), x_0\rangle +\int_0^1 (1-t)\langle x_0, \nabla^2 V(x+tx_0)\cdot x_0\rangle \, {\rm d}t\\
		&=\langle \nabla V(x), x_0\rangle +\frac{c}{2}\abs{x_0}^2+ \int_0^1 (1-t)\langle x_0, (\nabla^2 V(x+tx_0)-cI_n)\cdot x_0\rangle \, {\rm d}t.
	\end{align*}
	Now, we may proceed as above to prove that (i)$\Longleftrightarrow$(iii). 
\end{proof}

\begin{remark}\rm Let $E_0^V:={\rm Ker}(\nabla^2 V -cI_n)\subset \mathbb R^n$, and consider its orthogonal complement $E_\perp^V\subset \mathbb R^n$.  If $\tilde V(x)=V(x)-V(0)$, $x\in \mathbb R^n,$ by Proposition \ref{V-properties} one has a \textit{Pythagorean identity} for $\tilde V$: for every $x=x_0+x_\perp$ with $x_0\in E_0^V$ and $x_\perp\in E_\perp^V$ one has that $\tilde V(x)=\tilde V(x_0)+\tilde V(x_\perp).$
\end{remark}

\section{Bakry--\'Emery  log-Sobolev inequality: proof of Theorems \ref{Theorem-LSI} and \ref{Theorem-LSI-radial}}\label{section-3}

\subsection{Stability in Bakry--\'Emery  log-Sobolev inequality (Theorem \ref{Theorem-LSI})}\label{section-stability}

%
%
%

The proof relies on the adaptation of  Maurey's  arguments   \cite{Maurey} to deduce the estimate \eqref{stability-estimate-00} from the stability of the Pr\'ekopa--Leindler inequality for suitable functions. To do this, let us consider a 
positive log-concave  function $f\in W^{1, 2}(\RR^n, \mu_V)$; in particular, 
$f^2=e^g$ for a concave function $g:\RR^n\to \RR$. 
Now fix $\lambda\in (0, 1)$, and consider the following functions on $\RR^n$, given by
\begin{equation}\label{u-v-w}
	u_{\lambda}(x)=e^{g(x)/(1-\lambda)-V(x)}, \quad v(y)=e^{-V(y)}, \quad w_{\lambda}(z)=e^{g_{\lambda}(z)-V(z)},
\end{equation}
where $$g_{\lambda}(z)=\sup_{z=(1-\lambda)x+\lambda y}\{g(x)-[(1-\lambda)V(x)+\lambda V(y)-V((1-\lambda)x+\lambda y)]\}.$$
By definition, it follows that 
\begin{equation}\label{w_lambda}
	w_{\lambda}(z)=\sup_{z=(1-\lambda)x+\lambda y}u_{\lambda}(x)^{1-\lambda}v(y)^{\lambda}.
\end{equation}
Since $g$ is  concave and $x\mapsto V(x)-\frac{c}{2}|x|^2$ is convex, it is easily seen that $u_{\lambda}$ and $v$ are   log-concave. Moreover,  according to    \eqref{w_lambda} and B\"or\"oczky and De \cite[Lemma 7.3 (i)]{BoroczkyDe}, it follows that $w_{\lambda}$ is log-concave as well. 

Note that the functions $v$,  $u_\lambda$, and $w_\lambda$ belong to  $L^1(\mathbb R^n,{\rm d}x)$ for every $\lambda\in (0,1)$. Indeed, the case of $v$ is trivial. For the function $u_\lambda$,   if $l\in \partial g(0)$ (where $\partial g(0)$ stands for the superdifferential  of the concave function $g$ at the origin), it follows that
$$g(x)\leq g(0)+\langle l,x\rangle,\ \ \ x\in \mathbb R^n.$$ Moreover, since $\nabla^2 V -cI_n\geq 0$, we also have 
\begin{equation}\label{V-estimate}
	V(x)\geq V(0)+\langle \nabla V(0),x\rangle +\frac{c}{2}|x^2|,\ \ \ x\in \mathbb R^n.
\end{equation}
Therefore, for every $\lambda\in (0,1)$ and $x\in \mathbb R^n$, we deduce that $$\frac{g(x)}{1-\lambda}-V(x)\leq \frac{g(0)}{1-\lambda}-V(0)+\left(\frac{|l|}{1-\lambda}+|\nabla V(0)|\right)|x|-\frac{c}{2}|x^2|,$$
which implies $u_\lambda=e^{g/(1-\lambda)-V}\in L^1(\mathbb R^n,{\rm d}x).$ A similar reasoning also guarantees  that $w_\lambda\in L^1(\mathbb R^n,{\rm d}x)$ for every $\lambda\in (0,1)$. 

In view of the Pr\'ekopa--Leindler inequality, we define $$\varepsilon_{\lambda}:=\frac{\displaystyle \int_{\RR^n}w_{\lambda} \, {\rm d}x-\left(\int_{\RR^n}u_{\lambda} \, {\rm d}x\right)^{1-\lambda}}{\displaystyle \left(\int_{\RR^n}u_{\lambda} \, {\rm d}x\right)^{1-\lambda}}\ge 0.$$ Thus, $$\int_{\RR^n} w_{\lambda} \, {\rm d}x=(1+\varepsilon_{\lambda})\left(\int_{\RR^n}u_{\lambda} \, {\rm d}x\right)^{1-\lambda},$$ and the conditions  of Proposition \ref{Stability_nonsharp} are verified. In consequence, there exist a dimension-depending constant $C_n>0$ and $x_0^\lambda\in \RR^n$   for every $\lambda\in (0,1) $  such that 
\begin{equation}\label{main-stability-logSob}
	\int_{\RR^n} \abs{\frac{u_{\lambda}(x)}{\int_{\RR^n}u_{\lambda}}-v(x-x_0^\lambda)}\, {\rm d}x\le C_n \left(\frac{\varepsilon_{\lambda}}{\tau_{\lambda}}\right)^{\frac{1}{19}},
\end{equation} where $\tau_\lambda=\min(\lambda, 1-\lambda)$. 


Our aim is to take the limit in \eqref{main-stability-logSob} as $\lambda\to 0^+.$ To do this, we adapt the Maurey-argument  from  Bobkov and Ledoux \cite{BobkovLedoux}; in their case the function  $g: \mathbb R^n\to \mathbb R$ is smooth with compact support, while in ours  it is
concave. 
%

We first focus on the right-hand side of \eqref{main-stability-logSob}. 
By the definition of $g_\lambda$ and \eqref{V-strict-convex}, we have $$g_{\lambda}(z)\le \sup_{h\in \RR^n}\left[g\left(z+\frac{\lambda}{1-\lambda}h\right)-\frac{c\lambda}{2(1-\lambda)}|{h}|^2\right].$$
The function $g$ is concave, so it is $\mu-$a.e.\ differentiable; therefore, for $\mu-$a.e. $z\in \RR^n$ and for every $h\in \RR^n$, one has 
$$g\left(z+\frac{\lambda}{1-\lambda}h\right)\le g(z)+\frac{\lambda}{1-\lambda} \langle \nabla g(z), h\rangle.$$
As a consequence, for $\mu-$a.e. $z\in \RR^n$, we have 
$$g_{\lambda}(z)\le g(z)+\frac{\lambda}{1-\lambda}\sup_{h\in \RR^n}\left[\langle \nabla g(z), h\rangle -\frac{c}{2}|{h}|^2\right]=g(z)+\frac{\lambda}{2(1-\lambda)c}|{\nabla g(z)}|^2.$$
%
Accordingly, we deduce that 
$$\int_{\RR^n} w_{\lambda} \, {\rm d}x\le \int_{\RR^n}e^{g(z)+\frac{\lambda}{2c(1-\lambda)}|{\nabla g(z)}|^2}\, {\rm d}\mu_V(z),$$
which implies 
\begin{equation}\label{deficit-estimate}
	\frac{\varepsilon_{\lambda}}{\tau_{\lambda}}\le  \frac{\displaystyle \int_{\RR^n}e^{g(z)+\frac{\lambda}{2c(1-\lambda)}|{\nabla g(z)}|^2}\, {\rm d}\mu_V(z)-\left(\int_{\RR^n}u_{\lambda}\, {\rm d}x\right)^{1-\lambda}}{\tau_{\lambda}\displaystyle \left(\int_{\RR^n}u_{\lambda}\, {\rm d}x\right)^{1-\lambda}}.
\end{equation}
A simple computation shows that 
\begin{equation}\label{limit-1}
	\lim_{\lambda\to 0^+}\frac{{\rm d}}{{\rm d}\lambda}\left(\int_{\RR^n}e^{g(z)+\frac{\lambda}{2c(1-\lambda)}|{\nabla g(z)}|^2}\, {\rm d}\mu_V(z)\right)=\frac{1}{2c}\int_{\RR^n}e^{g(z)}|{\nabla g(z)}|^2\, {\rm d}\mu_V(z),
\end{equation}
and
\begin{equation}\label{limit-2}
	\lim_{\lambda \to 0^+}\frac{\rm d}{\rm d\lambda}\left(\int_{\RR^n}u_{\lambda} \, {\rm d}x\right)^{1-\lambda}={\rm Ent}_{\mu_V}(e^g)={\rm Ent}_{\mu_V}(f^2).
\end{equation} 
By using the dominated convergence theorem, we have 
\begin{equation}\label{limit-3}
	\lim_{\lambda \to 0^+}\left(\int_{\RR^n} u_{\lambda} \, {\rm d}x\right)=\int_{\RR^n}e^{g(x)-V(x)}\, {\rm d}x=\int_{\RR^n}e^g\, {\rm d}\mu_V=\int_{\RR^n}f^2 \, {\rm d}\mu_V=:\alpha.
\end{equation}
Thus, since $\displaystyle \lim_{\lambda\to 0^+}\frac{\lambda}{\tau_{\lambda}}=1$, in view of \eqref{limit-1}, \eqref{limit-2}, and \eqref{limit-3}, l'Hospital's  rule implies
\begin{align}\label{main-limit}
	\lim_{\lambda\to 0^+} \frac{\displaystyle  \int_{\RR^n}e^{g(z)+\frac{\lambda}{2c(1-\lambda)}|{\nabla g(z)}|^2}\, {\rm d}\mu_V(z)-\left(\int_{\RR^n}u_{\lambda}\, {\rm d}x\right)^{1-\lambda}}{\tau_{\lambda}\displaystyle \left(\int_{\RR^n}u_{\lambda}\, {\rm d}x\right)^{1-\lambda}}&=	\frac{1}{\alpha}\left(\frac{2}{c}\int_{\RR^n}|{\nabla f}|^2 \, {\rm d}\mu_V- {\rm Ent}_{\mu_V}(f^2)\right) \nonumber \\
	&=\delta^{\sf LSI}_{\mu_V}(f).
\end{align}
Taking the limit as $\lambda \to 0^+$ in \eqref{main-stability-logSob}, we obtain, according to \eqref{deficit-estimate} and \eqref{main-limit}, that 
\begin{equation}\label{Left-hand-side}
	\limsup_{\lambda \to 0^+}\int_{\RR^n} \abs{\frac{e^{g(x)/(1-\lambda)-V(x)}}{\int_{\RR^n}e^{g/(1-\lambda)}{\rm d}\mu_V}-e^{-V(x-x_0^\lambda)}}\, {\rm d}x\le C_n (\delta^{\sf LSI}_{\mu_V}(f))^{\frac{1}{19}}.
\end{equation}

Now, we focus on the left hand side of \eqref{main-stability-logSob}, which is the same as in \eqref{Left-hand-side}. Assume by contradiction, that there is a sequence $\lambda_k\to 0^+$ such that $|x_0^{\lambda_k}|\to \infty$. Due to \eqref{V-estimate}, it follows that $V(x-x_0^{\lambda_k})\to +\infty$ for every $x\in \mathbb R^n$; thus, by the dominated convergence theorem and \eqref{Left-hand-side}, it follows that 
\begin{eqnarray*}
	1&=&\lim_{k \to \infty}\int_{\RR^n} \abs{\frac{e^{g(x)/(1-\lambda_k)-V(x)}}{\int_{\RR^n}e^{g/(1-\lambda_k)}{\rm d}\mu_V}-e^{-V(x-x_0^{\lambda_k})}}\, {\rm d}x\leq\limsup_{\lambda \to 0^+}\int_{\RR^n} \abs{\frac{e^{g(x)/(1-\lambda)-V(x)}}{\int_{\RR^n}e^{g/(1-\lambda)}{\rm d}\mu_V}-e^{-V(x-x_0^\lambda)}}\, {\rm d}x\\&\le& C_n (\delta^{\sf LSI}_{\mu_V}(f))^{\frac{1}{19}},
\end{eqnarray*}
which contradicts  the assumption $\delta^{\sf LSI}_{\mu_V}(f)<C_n^{-19}.$ 

As a consequence, the set $\{x_0^\lambda:\lambda\in (0,1)\}\subset \mathbb R^n$ is bounded. Thus, we may consider a subsequence of $(x_0^\lambda)$ converging to some element $x_0\in \mathbb R^n$. A similar reasoning as above, based on  \eqref{Left-hand-side}, implies relation \eqref{stability-estimate-00}, concluding the proof.  

\subsection{Stability in radial Bakry--\'Emery  log-Sobolev inequality (Theorem \ref{Theorem-LSI-radial}/(i))}

Assume  that $f\in W^{1, 2}(\RR^n, \mu_V)$ is a positive log-concave radial function,  
the potential $V\in C^2(\mathbb R^n)$ is also radial,  and  $\nabla^2 V -cI_n\geq 0$ for some $c>0.$ In order to prove Theorem \ref{Theorem-LSI-radial}/(i), we show that the functions $u_{\lambda}, v$, and $w_{\lambda}$ introduced in \eqref{u-v-w} inherit the radial character  of $f$ and $V$. As before, let $e^g=f^2$; hence, $g$ is also radially symmetric, i.e., $g(\tau x)=g(x)$ for every $\tau\in O(n)$ and $x\in \mathbb R^n.$

Let us fix $\tau\in O(n)$ arbitrarily. First, we clearly have $v(\tau y)=e^{-V(\tau y)}=e^{-V(y)}=v(y)$ for every  $y\in \mathbb R^n;$ thus, $v$ is radial. Moreover, 
$u_{\lambda}(\tau x)=e^{g(\tau x)/(1-\lambda)-V(\tau x)}=e^{g(x)/(1-\lambda)-V(x)}=u_{\lambda}(x)$ for every  $x\in \mathbb R^n;$ therefore,  $u_{\lambda}$ is radial for every $\lambda\in (0,1)$.  Finally, the radiality of $w_{\lambda}$ follows at once by the same property of $g_{\lambda}$. Indeed, for  every $z\in \mathbb R^n,$  it follows that  
\begin{eqnarray*}
	g_{\lambda}(\tau z) 
	&=& V(\tau z)+\sup_{\tau z=(1-\lambda)x+\lambda y}\{g(x)-[(1-\lambda)V(x)+\lambda V(y)]\}\\                            
	&=&V(z)+ \sup_{\tau z=(1-\lambda)\tau \Tilde{x}+\lambda \tau \Tilde{y}}\{g(\tau \Tilde{x})-[(1-\lambda)V(\tau  \Tilde{x})+\lambda V(\tau \Tilde{y})]\} 
	\\
	&=&V(z)+\sup_{z=(1-\lambda)\Tilde{x}+\lambda \Tilde{y}}\{g(\Tilde{x})-[(1-\lambda)V(\Tilde{x})+\lambda V(\Tilde{y})]\}\\
	&=&g_{\lambda}(z).
\end{eqnarray*}

We are in the position to repeat the arguments from \S \ref{section-stability}, this time applying Corollary \ref{Stability_sharp-corollary} instead of Proposition \ref{Stability_nonsharp}. Therefore, by replacing the exponent $1/19$  by $1/2$ and setting $x_0=0$, we obtain $$\int_{\RR^n} \abs{\frac{f^2}{\alpha}-1}\, {\rm d}\mu_V \le C_n\left(\delta^{\sf LSI}_{\mu_V}(f)\right)^{\frac{1}{2}},$$
which is the required relation \eqref{stability-estimate-00-radial}.

\subsection{Optimality of the exponent 1/2 in \eqref{stability-estimate-00-radial} (Theorem \ref{Theorem-LSI-radial}/(ii))} Assume by contradiction that there exist an absolute constant $\eta>\frac{1}{2}$ and a dimension-depending number $C_n>0$  such that
\begin{equation}\label{contradiction-1}
	\int_{\RR^n} \abs{\frac{f^2}{\alpha}-1}\, {\rm d}\mu_V \le C_n\cdot\left(\delta^{\sf LSI}_{\mu_V}(f)\right)^\eta
\end{equation}
for every radial potential $V\in C^2(\mathbb R^n)$ with $\nabla^2 V -cI_n\geq 0$ for some $c>0$ and for every positive log-concave and radial function
$f\in W^{1, 2}(\RR^n, \mu_V)$. In particular,   \eqref{contradiction-1} is also valid for the quadratic function   $V(x)=\frac{c}{2}|x|^2+\frac{n}{2}\log(\frac{2\pi}{c})$; it is clear that ${\rm d}\mu_c:={\rm d}\mu_V =e^{-V}{\rm d}x$ is a probability measure on $\mathbb R^n.$ Similarly to Balogh and Krist\'aly \cite{Balogh-Kristaly-preprint}, for every $\epsilon>0$, let us consider the family of radially symmetric log-concave functions  $$f_\epsilon(x)=e^{-\frac{\epsilon}{2}|x|^2},\ x\in \mathbb R^n.$$
In particular, by \eqref{contradiction-1} one has for $0<\epsilon\ll 1$ that 
\begin{equation}\label{replaced-contra}
	\int_{\RR^n} \abs{\frac{e^{-{\epsilon}|x|^2}}{\int_{\mathbb R^n}f_\epsilon ^2 {\rm d}\mu_c}-1}\, {\rm d}\mu_c \le C_n\left(\delta^{\sf LSI}_{\mu_c}(f_\epsilon)\right)^\eta.
\end{equation}	
Simple computations show that
$$\int_{\mathbb R^n}f_\epsilon ^2 {\rm d}\mu_c=\left(\frac{2\epsilon}{c}+1\right)^{-\frac{n}{2}},\ \ \ \ \int_{\mathbb R^n} |\nabla f_\epsilon|^2\mathrm{~d} \mu_c=\epsilon^2\frac{n}{c}\left(\frac{2\epsilon}{c}+1\right)^{-\frac{n}{2}-1},$$
and
$$ \int_{\mathbb R^n} f_\epsilon^2 \log f_\epsilon^2 \mathrm{~d} \mu_c=-\epsilon\frac{n}{c}\left(\frac{2\epsilon}{c}+1\right)^{-\frac{n}{2}-1}.$$
Thus,
$$	{\rm Ent}_{\mu_c}(f_\epsilon^2)=\frac{n}{2}\left(\frac{2\epsilon}{c}+1\right)^{-\frac{n}{2}-1}\left(-\frac{2\epsilon}{c}+\left(\frac{2\epsilon}{c}+1\right)\log\left(\frac{2\epsilon}{c}+1\right)\right).$$
Dividing \eqref{replaced-contra} by $\epsilon>0,$ on the one hand, it follows by Fatou's lemma  that
\begin{eqnarray*}
	LHS&:=&\liminf_{\epsilon\to 0}\frac{1}{\epsilon}\int_{\RR^n} \abs{\frac{e^{-{\epsilon}|x|^2}}{\int_{\mathbb R^n}f_\epsilon ^2 {\rm d}\mu_c}-1}\, {\rm d}\mu_c = \liminf_{\epsilon\to 0}\frac{1}{\epsilon}\int_{\R^n}\Big| \left(\frac{2\epsilon}{c}+1\right)^{\frac{n}{2}}e^{-{\epsilon}|x|^2}-1\Big| \mathrm{d}\mu_c (x)\\&\geq& \int_{\R^n}\liminf_{\epsilon\to 0}\frac{1}{\epsilon}\Big| \left(\frac{2\epsilon}{c}+1\right)^{\frac{n}{2}}e^{-{\epsilon}|x|^2}-1\Big| \mathrm{d}\mu_c (x)\\&=&\int_{\R^n}\left| -|x|^2+\frac{n}{c} \right| \mathrm{d}\mu_c (x),
\end{eqnarray*}
which is positive and finite. On the other hand, we have 
$$\delta^{\sf LSI}_{\mu_c }(f_\epsilon):=\frac{\displaystyle\frac{2}{c}\int_{\RR^n} |{\nabla f_\epsilon}|^2\, {\rm d}\mu_c-{\rm Ent}_{\mu_c}(f_\epsilon^2)}{\displaystyle\int_{\mathbb R^n} f_\epsilon^2{\rm d}\mu_c}=\frac{n}{2}\left(\frac{2\epsilon}{c}-\log\left(\frac{2\epsilon}{c}+1\right)\right)=n\frac{\epsilon^2}{c^2}+o(\epsilon^2), \ \ 0<\epsilon\ll 1.$$
Thus, by the assumption $\eta>\frac{1}{2}$, it follows that $$RHS:=\lim_{\epsilon\to 0}\frac{\left(\delta^{\sf LSI}_{\mu_c}(f_\epsilon)\right)^\eta}{\epsilon}=0,$$
a contradiction.

	%

\section{Talagrand transport inequality: proof of Theorems \ref{Theorem-Talagrand} and \ref{Theorem-Talagrand-radial}}\label{section-4}

\subsection{Stability in Talagrand's inequality (Theorem \ref{Theorem-Talagrand}/(i))}\label{section-talag-1}

%
%

Let  $\nu $ be a probability measure on $\mathbb R^n$ which is absolutely continuous with respect to $\mu_V$, having its density
$u=\frac{{\rm d}\nu}{{\rm d}\mu_V}.$
Given the cost function $c(x,y)=\frac{c}{2}\abs{x-y}^2$,  the optimal mass transport theory -- developed by Brenier \cite{Brenier} and McCann \cite{McCann} -- guarantees the existence of a function $g:\mathbb R^n\to \mathbb R$ such that $x\mapsto \frac{c}{2}|x|^2+g(x)$ is convex on $\mathbb R^n$ and 
$T(x)=x+\frac{1}{c}\nabla g(x)$ optimally pushes $\mu_V$ forward onto $\nu$, see also Villani \cite[Theorem  2.44]{Villani-short}.
Consider the $c$-transform  of $(-g)$ with respect to $c(x,y)$, namely, the inf-convolution function $$Lg:\RR^n\to \RR,\quad  Lg(y)=\inf_{x\in \RR^n}\left[g(x)+\frac{c}{2}\abs{x-y}^2\right],$$
which is the main tool in    the Kantorovich duality theory.

For every $\lambda\in (0,1)$, we define the following functions on $\RR^n$:
\begin{equation}\label{u-v-w-talag}
	u_{\lambda}(x)=e^{-\lambda g(x)-V(x)}, \quad v_{\lambda}(y)=e^{(1-\lambda)Lg(y)-V(y)}, \quad w(z)=e^{-V(z)}.
\end{equation}
We claim that $u_{\lambda}$, $v_{\lambda}$, and $w$ belong to $L^1(\mathbb R^n,{\rm d}x)$. Indeed, since $x\mapsto \frac{c}{2}|x|^2+g(x)$ is convex, it follows that for some $l\in \mathbb R^n$, one has


$$\frac{c}{2}|x|^2+g(x) \geq g(0)+\langle l,x\rangle,\ \ x\in \mathbb R^n,$$
while, by definition, 
$$Lg(y)\leq g(0)+\frac{c}{2}\abs{y}^2,\ \  y\in \mathbb R^n.$$ Therefore, by \eqref{V-estimate}, for every $\lambda\in (0,1)$ it follows that
\begin{equation}\label{u-lambda-needed}
	-\lambda g(x)-V(x)\leq -(1-\lambda)\frac{c}{2}|x|^2+\langle \lambda l+\nabla V(0),x\rangle -\lambda g(0)-V(0),\ \ x\in \mathbb R^n,
\end{equation}
and 
$$(1-\lambda)Lg(y)-V(y)\leq -\lambda\frac{c}{2}|y|^2-\langle \nabla V(0),y\rangle +(1-\lambda) g(0)-V(0),\ \ y\in \mathbb R^n,$$
which prove the claim.

According to \eqref{V-strict-convex}, it is straightforward to see that for every $z\in \mathbb R^n$ we have
$$w(z)\geq \sup_{z=(1-\lambda)x+\lambda y}u_{\lambda}(x)^{1-\lambda}v(y)^{\lambda}.$$
Thus, in view of the Pr\'ekopa--Leindler inequality   we consider 
$$\varepsilon_{\lambda}:=\frac{1-\left(\displaystyle \int_{\RR^n} e^{-\lambda g}\, {\rm d}\mu_V\right)^{1-\lambda}\left(\displaystyle \int_{\RR^n}e^{(1-\lambda)Lg}\, {\rm d}\mu_V\right)^{\lambda}}{\left(\displaystyle \int_{\RR^n} e^{-\lambda g}\, {\rm d}\mu_V\right)^{1-\lambda}\left(\displaystyle \int_{\RR^n}e^{(1-\lambda)Lg}\, {\rm d}\mu_V\right)^{\lambda}}\ge 0.$$
Then $$\int_{\RR^n}w \, {\rm d}x=(1+\varepsilon_{\lambda})\left(\int_{\RR^n} u_{\lambda}\, {\rm d}x\right)^{1-\lambda}\left(\int_{\RR^n}v_{\lambda}\, {\rm d}x\right)^{\lambda},$$ and  since $w$ is obviously log-concave, we may apply Proposition \ref{Stability_nonsharp} to conclude that there exist a dimensional constant $C_n>0$ and a translation vector $x_0^\lambda\in \RR^n$ for every $\lambda\in (0,1)$ such that 
\begin{equation}\label{stability-ineq-Talagrand}
	\int_{\RR^n}\abs{\frac{u_{\lambda}(x)}{\int_{\RR^n}u_{\lambda}}-\frac{v_{\lambda}(x-x_0^\lambda)}{\int_{\RR^n}v_{\lambda}}}\, {\rm d}x\le C_n \left(\frac{\varepsilon_{\lambda}}{\tau_{\lambda}}\right)^{\frac{1}{19}},
\end{equation}
where $\tau_{\lambda}=\min(\lambda, 1-\lambda)$.

On the one hand, applying the dominated convergence theorem, we have  
$\lim_{\lambda \to 0^+}\varepsilon_{\lambda}=0.$
A simple computation shows that 
\begin{align*}
	\lim_{\lambda \to 0^+}\frac{\rm d}{\rm d\lambda}\left(\int_{\RR^n} e^{-\lambda g}\, {\rm d}\mu_V\right)^{1-\lambda}&=
	-\int_{\RR^n} g\, {\rm d}\mu_V
\end{align*}
and \begin{align*}
	\lim_{\lambda \to 0^+}\frac{\rm d}{{\rm d} \lambda}\left(\int_{\RR^n}e^{(1-\lambda)Lg}\, {\rm d}\mu_V\right)^{\lambda}&=
	\log\left(\int_{\RR^n}e^{Lg}\, {\rm d}\mu_V\right).
\end{align*}
Accordingly,  we have 
\begin{align*}
	\lim_{\lambda\to 0^+}\frac{\varepsilon_{\lambda}}{\lambda}&=\lim_{\lambda\to 0^+}\frac{1-\left(\displaystyle \int_{\RR^n} e^{-\lambda g}\, {\rm d}\mu_V\right)^{1-\lambda}\left(\displaystyle \int_{\RR^n}e^{(1-\lambda)Lg}\, {\rm d}\mu_V\right)^{\lambda}}{\lambda}
	=
	\int_{\RR^n}g\, {\rm d}\mu_V-\log\left(\int_{\RR^n}e^{Lg}\, {\rm d}\mu_V\right).
\end{align*}
Therefore, it follows that
\begin{equation}\label{limit-infconv}
	\lim_{\lambda \to 0^+}\frac{\varepsilon_{\lambda}}{\tau_{\lambda}}=\int_{\RR^n}g\, {\rm d}\mu_V-\log\left(\int_{\RR^n}e^{Lg}\, {\rm d}\mu_V\right).
\end{equation}

By construction,  the pair $(g, Lg)$ solves the dual Kantorovich problem, thus  
\begin{equation}\label{kantorovich-optimal}
	\int_{\RR^n} Lg\, {\rm d}\nu-\int_{\RR^n}g\, {\rm d}\mu_V=\frac{c}{2}W_2^2(  \nu,\mu_V),
\end{equation}
see, e.g.,\ Villani \cite[Chapter 2]{Villani-short}. Moreover, recall  the variational characterization of the entropy, i.e. \ 
\begin{equation}\label{entropy-variational}
	{\rm Ent}_{\mu_V}(u)=\sup \int_{\RR^n} uv\, {\rm d}\mu_V
\end{equation}
for every $\mu_V-$integrable function $u:\RR^n\to \RR$, where the supremum is taken over all measurable functions $v$ with $\displaystyle \int_{\RR^n}e^v\, {\rm d}\mu_V\le 1$, see, e.g.,\ Bobkov and Ledoux \cite{BobkovLedoux}.

Setting $u=\frac{{\rm d} \nu}{{\rm d} \mu_V}$  
and  $v =Lg -\log\left(\displaystyle\int_{\RR^n}e^{Lg}\, {\rm d}\mu_V\right)$  in \eqref{entropy-variational},  it turns out that 
\begin{equation}\label{entropy-kant}
	\int_{\RR^n}Lg\, {\rm d}\nu-\log\left(\int_{\RR^n}e^{Lg}\, {\rm d}\mu_V\right)\le{\rm Ent}_{\mu_V}(u)= D(\nu||\mu_V).
\end{equation}
Combining the latter estimate with \eqref{kantorovich-optimal}, we obtain   
$$\int_{\RR^n} g\, {\rm d}\mu_V-\log\left(\int_{\RR^n}e^{Lg}\, {\rm d}\mu_V\right)\le D(\nu||\mu_V)-\frac{c}{2}W_2^2(\nu,\mu_V)=\delta^{\sf Tal}_{\mu_V}(\nu).$$
By \eqref{stability-ineq-Talagrand}, \eqref{limit-infconv}, the latter relation, and a change of variables, we have
\begin{equation}\label{almost-final}
	\limsup_{\lambda\to 0^+}	\int_{\RR^n}\abs{\frac{u_{\lambda}(x+x_0^\lambda)}{\int_{\RR^n}u_{\lambda}}-\frac{v_{\lambda}(x)}{\int_{\RR^n}v_{\lambda}}}\, {\rm d}x\le C_n \left(\delta^{\sf Tal}_{\mu_V}(\nu)\right)^{\frac{1}{19}}.
\end{equation}

Let us assume by contradiction that there is a sequence $\lambda_k\to 0^+$ such that $|x_0^{\lambda_k}|\to \infty$. Due to \eqref{u-lambda-needed}, it follows that for every $x\in \mathbb R^n$ one has $	-\lambda_k g(x+x_0^k)-V(x+x_0^k)\to -\infty$, thus $u_{\lambda_k}(x+x_0^{\lambda_k})\to 0$ as $k\to \infty$. 
Therefore, by the dominated convergence theorem and \eqref{almost-final}, we get 
\begin{eqnarray*}
	1&=&\lim_{k \to \infty} \int_{\RR^n}\abs{\frac{u_{\lambda_k}(x+x_0^{\lambda_k})}{\int_{\RR^n}u_{\lambda_k}}-\frac{v_{\lambda_k}(x)}{\int_{\RR^n}v_{\lambda_k}}}\, {\rm d}x\leq \limsup_{\lambda\to 0^+}	\int_{\RR^n}\abs{\frac{u_{\lambda}(x+x_0^\lambda)}{\int_{\RR^n}u_{\lambda}}-\frac{v_{\lambda}(x)}{\int_{\RR^n}v_{\lambda}}}\, {\rm d}x\\&\le& C_n \left(\delta^{\sf Tal}_{\mu_V}(\nu)\right)^{\frac{1}{19}},
\end{eqnarray*}
contradicting  the assumption $\delta^{\sf Tal}_{\mu_V}(\nu)<C_n^{-19}.$ Accordingly, the set $\{x_0^\lambda:\lambda\in (0,1)\}\subset \mathbb R^n$ is bounded. By extracting a subsequence of $(x_0^\lambda)$ that converges to some $x_0\in \mathbb R^n$, \eqref{almost-final} provides the  estimate 
$$
\int_{\RR^n}\abs{e^{-V(x+x_0)}-\frac{e^{Lg(x)-V(x)}}{\int_{\RR^n}e^{Lg}\, {\rm d}\mu_V}}\, {\rm d}x\le C_n \left(\delta^{\sf Tal}_{\mu_V}(\nu)\right)^{\frac{1}{19}}.
$$
which is equivalent to  \eqref{deficit-estimate-talagrand}.

\subsection{Equality in Talagrand transport inequality  (Theorem \ref{Theorem-Talagrand}/(ii))} Assume that there is equality in \eqref{Talagrand-inequality} for some probability measure $\nu \ll \mu_V$ on $\mathbb R^n$, i.e., $\delta^{\sf Tal}_{\mu_V}(\nu)=0.$ First, by \eqref{deficit-estimate-talagrand}, we have  
\begin{equation}\label{equality-1}
	e^{Lg(x-x_0)+V(x)-V(x-x_0)}=\int_{\RR^n}e^{Lg}\, {\rm d}\mu_V\ \ {\rm for}\ \ \mu_V-{\rm a.e.}\ x\in \mathbb R^n,
\end{equation}
where  $g:\mathbb R^n\to \mathbb R$ is such that $T(x):=x+\frac{1}{c}\nabla g(x)$ becomes the Brenier map pushing forward  $\mu_V$ onto $\nu $. 
Moreover, tracking back the estimates in the previous proof, it turns out that equality should also hold  in \eqref{entropy-kant}, i.e., 
$$\int_{\RR^n}Lg\, {\rm d}\nu-\log\left(\int_{\RR^n}e^{Lg}\, {\rm d}\mu_V\right)= D(\nu||\mu_V).$$ The latter relation is precisely the equality case in the  Donsker--Varadhan variational formula  \cite[Theorem 5.2]{Donsker-Varadhan}, which is characterized by 
\begin{equation}\label{equality-2}
	\frac{{\rm d} \nu}{{\rm d} \mu_V}(x)=\frac{e^{Lg(x)}}{\int_{\RR^n}e^{Lg}\, {\rm d}\mu_V}\ \ \ {\rm for}\ \ \mu_V-{\rm a.e.}\ x\in \mathbb R^n.
\end{equation}
Therefore, combining relations \eqref{equality-1} and \eqref{equality-2}, it follows  for $\mu_V$-a.e.\ $x\in \mathbb R^n$ that 
$$\frac{{\rm d} \nu}{{\rm d} x}(x-x_0)=\left(\frac{{\rm d} \nu}{{\rm d} \mu_V}\cdot\frac{{\rm d} \mu_V}{{\rm d} x}\right)(x-x_0)=\frac{e^{Lg(x-x_0)}}{\int_{\RR^n}e^{Lg}\, {\rm d}\mu_V}\cdot e^{-V(x-x_0)}=e^{-V(x)}=\frac{{\rm d} \mu_V}{{\rm d} x}(x),$$
i.e.,  the density of $\nu$ is a translation of the density of $\mu_V$. In particular, the optimal transport map pushing forward $\mu_V$ into  $\nu$  is the translation $T(x)=x-x_0,$ $x\in \mathbb R^n$. Hence, 
\begin{equation}\label{Wass-relation}
	W_2^2(\nu,\mu_V)=\int_{\mathbb R^n}|x-T(x)|^2{\rm d}\mu_V(x)=|x_0|^2.
\end{equation}
Again by  relations \eqref{equality-1} and \eqref{equality-2}, it follows that
$$\frac{{\rm d} \nu}{{\rm d} \mu_V}(x)=e^{V(x)-V(x+x_0)}\ \ \ {\rm for}\ \ \mu_V-{\rm a.e.}\ x\in \mathbb R^n.$$
Accordingly, one has 
\begin{eqnarray*}
	D(\nu || \mu_V)&=&\int_{\RR^n} \frac{{\rm d}\nu}{{\rm d}\mu_V}\log \frac{{\rm d}\nu}{{\rm d}\mu_V} \, {\rm d}\mu_V=\int_{\RR^n} e^{-V(x+x_0)}({V(x)-V(x+x_0)}) \, {\rm d}x\\&=&\int_{\RR^n} ({V(x-x_0)-V(x)}) \, {\rm d}\mu_V(x).
\end{eqnarray*}
Note that by a Taylor expansion, we have 
\begin{align}\label{V-expans-0000}
	\nonumber	V(x-x_0)-V(x)&=-\langle \nabla V(x), x_0\rangle +\int_0^1 (1-t)\langle x_0, \nabla^2 V(x-tx_0)\cdot x_0\rangle \, {\rm d}t\\
	&=-\langle \nabla V(x), x_0\rangle +\frac{c}{2}\abs{x_0}^2+ \int_0^1 (1-t)\langle x_0, (\nabla^2 V(x-tx_0)-cI_n)\cdot x_0\rangle \, {\rm d}t.
\end{align}
Furthermore, by the fast decaying behavior of $e^{-V}$ at infinity (see  \eqref{V-estimate}), it follows that $$-\int_{\RR^n} \langle \nabla V(x), x_0 \rangle \, {\rm d}\mu_V(x)=\int_{\RR^n}\langle \nabla (e^{-V(x)}), x_0\rangle\, {\rm d}x=\int_{\RR^n}  {\rm div}(e^{-V(x)}   x_0 ) {\rm d}x=0,$$
thus
\begin{equation}\label{taylor-term-00000}
	\int_{\RR^n} (V(x-x_0)-V(x))\, {\rm d}\mu_V(x)=\frac{c}{2}\abs{x_0}^2 +\int_0^1 (1-t)\int_{\RR^n} \left\langle x_0, (\nabla^2 V(x-tx_0)-cI_n)\cdot x_0 \right \rangle\, {\rm d}\mu_V(x)\, {\rm d}t.
\end{equation}
%
%
%
%
%
Since $\delta^{\sf Tal}_{\mu_V}(\nu)=0,$ the latter relation and \eqref{Wass-relation} imply 
$$\int_0^1 (1-t)\int_{\RR^n} \left\langle x_0, (\nabla^2 V(x-tx_0)-cI_n)\cdot x_0 \right \rangle\, {\rm d}\mu_V(x)\, {\rm d}t=0$$
for every $x\in \mathbb R^n.$ Since $\nabla^2 V -cI_n\geq 0$, it follows that   
$$\nabla^2 V(x)\cdot x_0=cx_0, \qquad \forall x\in \RR^n,$$
i.e.,  $x_0\in  {\rm Ker}(\nabla^2 V -cI_n)$. 

The converse is trivial, based on Proposition \ref{V-properties}. 

\begin{remark}\rm 
	Since  $T(x)=x+\frac{1}{c}\nabla g(x)$, it follows that  $\nabla g(x)=-cx_0$, i.e., $g(x)=-c\langle x_0,x\rangle +c_0$ and $Lg(y)=-c\langle x_0,y\rangle- \frac{c}{2}|x_0|^2 +c_0$ for some $c_0\in \mathbb R^n$. 
\end{remark}

Having the equality case in the Talagrand transport inequality, we can identify the equality in the Bakry--\'Emery log-Sobolev inequality as well, without making any concavity assumption on the extremizer. More precisely, we have:

\begin{corollary}\label{corollary-LSI=}
  Let ${\rm d}\mu_V=e^{-V}{\rm d}x$  be a probability measure on $\mathbb R^n$ such that $V\in C^2(\mathbb R^n)$ and  $\nabla^2 V -cI_n\geq 0$ for some $c>0.$ Then  equality holds in \eqref{gaussian-logsob-00} for some $f\in W^{1, 2}(\RR^n, \mu_V)$  if and only if $f(x)=a\cdot e^{\langle x_0,x\rangle}$ for some $a\in \mathbb R$  and $x_0\in  {\rm Ker}(\nabla^2 V -cI_n)$.
  \end{corollary}

\begin{proof}
	We first recall the HWI inequality of Otto and Villani \cite{Otto-Villani} (see also Cordero-Erausquin \cite{Cordero-ARMA}), which establishes a relation between the entropy, Wasserstein distance, and   Fisher information: given the probability measures ${\rm d}\mu_V=e^{-V}{\rm d}x$ and ${\rm d}\nu= f_0^2{\rm d}\mu_V$  with $f_0\in W^{1,2}(\mathbb R^n,{\rm d}\mu_V)$, one has  
	$${\rm Ent}_{\mu_V}(f_0^2)\leq W_2(\mu_V,f_0^2{\rm d}\mu_V)\sqrt{I_{\mu_V}(f_0^2)}-\frac{c}{2}W_2^2(\mu_V,f_0^2{\rm d}\mu_V).$$
	Here, as usual, 
	$$I_{\mu_V}(f_0^2)=4\int_{\mathbb R^n}|\nabla f_0|^2{\rm d}\mu_V$$
	stands for the Fisher information of $f_0^2$ with respect to $\mu_V$. A suitable rearrangement of the HWI inequality and elementary algebraic manipulations  yield  that  the \textit{HWI deficit} can be written in terms of the Bakry--\'Emery log-Sobolev  and Talagrand deficits  as follows:
	\begin{eqnarray*}
		0&\leq& \delta^{\sf HWI}(f_0^2{\rm d}\mu_V):=W_2^2(\mu_V,f_0^2{\rm d}\mu_V)\, I_{\mu_V}(f_0^2) -\left({\rm Ent}_{\mu_V}(f_0^2)+\frac{c}{2}W_2^2(\mu_V,f_0^2{\rm d}\mu_V)\right)^2\\&=&2c W_2^2(\mu_V,f_0^2{\rm d}\mu_V)\, \delta^{\sf LSI}_{\mu_V}(f_0) -\delta^{\sf Tal}_{\mu_V}(f_0^2{\rm d}\mu_V).   
	\end{eqnarray*}
As a result, we have  
\begin{equation}\label{relation-deficits}
	\delta^{\sf Tal}_{\mu_V}(f_0^2{\rm d}\mu_V)\leq 2c W_2^2(\mu_V,f_0^2{\rm d}\mu_V)\, \delta^{\sf LSI}_{\mu_V}(f_0).
\end{equation}

Now, assume that equality holds in \eqref{gaussian-logsob-00} for some $f\in W^{1, 2}(\RR^n, \mu_V)\setminus \{0\}$, i.e.,  $\delta^{\sf LSI}_{\mu_V}(f)=0$. If $f_0=|f|/\sqrt{\alpha}$, where, as usual, $\alpha=\displaystyle \int_{\mathbb R^n} f^2{\rm d}\mu_V$, it follows that $$\delta^{\sf LSI}_{\mu_V}(f_0)=0.$$ According to \eqref{relation-deficits}, it turns out that $\delta^{\sf Tal}_{\mu_V}(f_0^2{\rm d}\mu_V)=0.$
By applying Theorem  \ref{Theorem-Talagrand}/(ii), we have $f_0^2{\rm d}\mu_V=e^{-V(\cdot +x_0)}{\rm d}x$, where  $x_0\in  {\rm Ker}(\nabla^2 V -cI_n)$. Thus, $f^2=\alpha e^{V(x)-V(x +x_0)}$ for a.e. $x\in \mathbb R^n$. By Proposition \ref{V-properties}, it follows that $f(x)=a\cdot e^{-c \langle x_0,x\rangle/2}$ for  $a=\pm \sqrt{\alpha}e^{(V(0)-V(x_0))/2}\in \mathbb R$. Clearly, $\tilde x_0=-cx_0/2$ belongs to ${\rm Ker}(\nabla^2 V -cI_n)$,
which concludes the proof.

For the  converse statement, assume that $f(x)=a\cdot e^{ \langle x_0,x\rangle}$ for some $a\in \mathbb R$ and $x_0\in  {\rm Ker}(\nabla^2 V -cI_n)$. Then, if $\overline x_0=\frac{2}{c}x_0$, by taking Proposition \ref{V-properties} into account, a simple computation yields 
$${\rm Ent}_{\mu_V}(f^2)=a^2\frac{c}{2}|\overline  x_0|^2e^{-V(\overline  x_0)+V(0)+c|\overline  x_0|^2}\ \  {\rm and}\ \ \int_{\RR^n} |{\nabla f}|^2\, {\rm d}\mu_V=a^2|x_0|^2e^{-V(\overline  x_0)+V(0)+c|\overline  x_0|^2},$$
which concludes the proof. 
\end{proof}

\begin{remark}\rm\label{remark-equal}
The equality in \eqref{gaussian-logsob-00} has been analytically   characterized in the  paper of Arnold,   Markowich,  Toscani,  and Unterreiter \cite[Theorem 3.11]{CPDE-equality}. According to them, equality holds in \eqref{gaussian-logsob-00} for some positive $f\in  W^{1, 2}(\RR^n, \mu_V)$ if and only if there exist Cartesian coordinates $y=(y_1,...,y_n)=y(x)$ on $\mathbb R^n$ and a function $B:\mathbb R^{n-1}\to \mathbb R$ with the following properties:
\begin{itemize}
	\item[(i)] for some $\beta \in \mathbb R$ one has 
	$$V(x(y))=\frac{c}{2}y_1^2+\beta y_1 +B(y_2,...,y_n),$$
	\item[(ii)] for some $\xi\in \mathbb R$ one has
$$f=\frac{{\rm d}\rho}{{\rm d}e^{-V}}\ \ {\rm with}\ \ \rho=e^{-V(x(y))+\xi y_1-\frac{\xi^2}{2c}+\frac{\beta \xi}{c}}.$$
\end{itemize}
In the sequel, we prove that 
the characterization of the equality in Corollary \ref{corollary-LSI=}  is an alternative, algebraic reformulation  of   (i)\&(ii). To see this, let $x_0\in  {\rm Ker}(\nabla^2 V -cI_n)$. In particular, this implies that the function   $x\mapsto \langle\nabla V(x),x_0\rangle -c\langle x,x_0\rangle$ is constant, i.e., \begin{equation}\label{V-relation}
	\langle\nabla V(x),x_0\rangle -c\langle x,x_0\rangle=\langle\nabla V(0),x_0\rangle, \quad x\in \mathbb R^n.
\end{equation}   For simplicity, we assume that $x_0\neq 0$. If $e_1=x_0/|x_0|$, we complete $e_1$ to an orthonormal basis $\{e_1,...,e_n\}$ of $\mathbb R^n$. We consider the Cartesian coordinates $y_i(x)=\langle e_i,x\rangle$, $i\in \{1,...,n\}$ on $\mathbb R^n$ and $x(y)=\sum_{i=1}^n y_ie_i.$ In these coordinates,   relation \eqref{V-relation} implies that
$$\frac{\partial}{\partial y_1}V(x(y))=\langle\nabla V(x(y)),e_1\rangle=cy_1+\beta,  $$
where $\beta=\langle\nabla V(0),x_0\rangle/|x_0|$. Therefore, an integration with respect to $y_1$ yields (i) from above. 

On the other hand, since $f(x)=a\cdot e^{\langle x_0,x\rangle}$ for some (or, in fact, for any) $a>0$, see  Corollary \ref{corollary-LSI=}, it turns out that
$$\rho(x(y))=e^{-V(x(y))}f(x(y))=e^{-V(x(y))+\xi y_1-\frac{\xi^2}{2c}+\frac{\beta \xi}{c}},$$
with the identification $\xi=|x_0|$ and $\log a= -\frac{\xi^2}{2c}+\frac{\beta \xi}{c}.$

\end{remark}

We conclude this subsection with examples of potentials $V$ (and their corresponding null spaces) appearing in our results, illustrating how the algebraic representation in Corollary \ref{corollary-LSI=} reduces the problem of the equality case to simple eigenvector calculations.   

\begin{example}\rm \label{example-1}
	Let $c>0$ and $V:\mathbb R^n\to \mathbb R$ be defined by  $V(x)=\frac{c}{2}|x|^2+\sum_{i=1}^n a_i x_i^4$, where $a_i\geq 0$, $i\in \{1,...,n\}$. Then $\nabla^2 V -cI_n\geq 0$ and  $x_0=(x_1^0,...,x_n^0)\in \mathbb R^n$ belongs to ${\rm Ker}(\nabla^2 V -cI_n)$ if and only if $a_ix_i^0=0$ for every $i\in \{1,...,n\}$. 
\end{example}

\begin{example}\rm\label{example-2}
	Let $c>0$ and $V:\mathbb R^n\to \mathbb R$ be given by  $V(x)=\frac{c}{2}|x|^2+e^{\langle a,x\rangle}$, where $a\in \mathbb R^n$. Then $\nabla^2 V -cI_n\geq 0$ and  $x_0\in {\rm Ker}(\nabla^2 V -cI_n)$ if and only if $\langle a,x_0\rangle=0$. 
\end{example}

\subsection{Stability in radial Talagrand transport inequality (Theorem \ref{Theorem-Talagrand-radial}/(i))} Since $\mu_V$ and $\nu $ have radially symmetric densities, it follows that both the Kantorovich potential $g:\mathbb R^n\to \mathbb R$ and the  Brenier map $T(x)=x+\frac{1}{c}\nabla g(x)$ are also radially symmetric. In a similar manner as in the proof of Theorem \ref{Theorem-LSI-radial}/(i), one can prove that the inf-convolution $Lg$ is also radial. As a consequence, the functions  in \eqref{u-v-w-talag} are radial. 

Now we can follow the arguments from \S\ref{section-talag-1}; therefore, applying Corollary \ref{Stability_sharp-corollary}, it turns  out that  $$\int_{\RR^n}\abs{1-\frac{e^{Lg(x)}}{\int_{\RR^n}e^{Lg}\, {\rm d}\mu_V}}\, {\rm d}\mu_V\le C_n\cdot \left(\delta^{\sf Tal}_{\mu_V}(\nu)\right)^{\frac{1}{2}},$$
which is precisely the estimate \eqref{deficit-estimate-talagrand-radial}. 

\subsection{Optimality of the exponent 1/2 in \eqref{deficit-estimate-talagrand-radial} (Theorem \ref{Theorem-Talagrand-radial}/(ii))} The argument is similar to the one in the proof of Theorem \ref{Theorem-LSI-radial}/(ii).   Assume by contradiction that there exists an absolute constant $\eta>\frac{1}{2}$ and a dimension-depending number $C_n>0$  such that
\begin{equation}\label{contradiction-1-t}
	\int_{\RR^n}\abs{1-\frac{e^{Lg(x)}}{\int_{\RR^n}e^{Lg}\, {\rm d}\mu_V}}\, {\rm d}\mu_V\le C_n\cdot \left(\delta^{\sf Tal}_{\mu_V}(\nu)\right)^\eta
\end{equation}
for every radial potential $V\in C^2(\mathbb R^n)$ with $\nabla^2 V -cI_n\geq 0$ for some $c>0$ and every probability measure $\nu\ll \mu_V$ on $\mathbb R^n$ with radial density. Therefore,   \eqref{contradiction-1-t} is also valid for the quadratic function   $V(x)=\frac{c}{2}|x|^2+\frac{n}{2}\log(\frac{2\pi}{c})$.  Let ${\rm d}\mu_c:={\rm d}\mu_V =e^{-V}{\rm d}x$ be the probability measure on $\mathbb R^n$ associated with this potential $V$.

For every $\epsilon>0$, we consider the probability measure  $\nu_\epsilon$ on $\mathbb R^n$ by pushing forward   $\mu_c=\mu_V$ by means of the map $T_\epsilon(x)=\frac{x}{\sqrt{1+\epsilon/c}}$, $x\in \mathbb R^n$. Due to \eqref{contradiction-1-t}, one has for every $0<\epsilon\ll 1$ that   
\begin{equation}\label{contradiction-2-t}
	\int_{\RR^n}\abs{1-\frac{e^{Lg_\epsilon(x)}}{\int_{\RR^n}e^{Lg_\epsilon}\, {\rm d}\mu_c}}\, {\rm d}\mu_c\le C_n\cdot \left(\delta^{\sf Tal}_{\mu_c}(\nu_\epsilon)\right)^\eta,
\end{equation}
where $g_\epsilon:\mathbb R^n\to \mathbb R$ is the Kantorovich potential in the optimal map  $T_\epsilon(x):=x+\frac{1}{c}\nabla g_\epsilon(x)$   by pushing forward  $\mu_c $ into $\nu_\epsilon.$ Accordingly, $\nabla g_\epsilon(x)=c c_\epsilon x$, where $c_\epsilon=-1+\sqrt{\frac{c}{\epsilon+c}}$; thus
$g_\epsilon(x)=\frac{c}{2}c_\epsilon |x|^2+c_0$ for some $c_0\in \mathbb R.$   Moreover, 
$$\frac{{\rm d}\nu_\epsilon}{{\rm d}\mu_c}(x)=(1+\epsilon/c)^\frac{n}{2}e^{-\frac{\epsilon}{2}|x|^2}, \ \ x\in \mathbb R^n,$$
which implies 
$$D(\nu_\epsilon||\mu_c)=\int_{\RR^n} \frac{{\rm d}\nu_\epsilon}{{\rm d}\mu_c}\log \frac{{\rm d}\nu_\epsilon}{{\rm d}\mu_c} \, {\rm d}\mu_c=\frac{n}{2}\left(\log\left(1+\frac{\epsilon}{c}\right)-\frac{\epsilon}{\epsilon+c}\right).$$ A similar computation shows that
$$	W_2^2(\nu_\epsilon,\mu_c )=\int_{\mathbb R^n}|x-T_\epsilon(x)|^2{\rm d}\mu_c(x)=c_\epsilon^2\int_{\mathbb R^n}|x|^2{\rm d}\mu_c(x)=c_\epsilon^2\frac{n}{c}.$$
Combining these relations, one has for $0<\epsilon\ll 1$ that
\begin{eqnarray}\label{first-contr}
	\nonumber	\delta^{\sf Tal}_{\mu_c}(\nu_\epsilon)&=&D(\nu_\epsilon||\mu_c)-\frac{c}{2}W_2^2(\nu_\epsilon,\mu_c )=\frac{n}{2}\left(\log\left(1+\frac{\epsilon}{c}\right)-\frac{\epsilon}{\epsilon+c}-c_\epsilon^2\right)\\&=&\frac{n}{8c^2}\epsilon^2 +o(\epsilon^2).
\end{eqnarray}

On the other hand,   $Lg_\epsilon(x)=\frac{c c_\epsilon}{2(1+c_\epsilon)}|x|^2+c_0, $ $x\in \mathbb R^n,$  and 
$$\int_{\RR^n}e^{Lg_\epsilon}\, {\rm d}\mu_c=(1+c_\epsilon)^\frac{n}{2}e^{c_0}=\left(\frac{c}{\epsilon+c}\right)^\frac{n}{4}e^{c_0}.$$
Moreover, Fatou's lemma implies 
\begin{eqnarray*}
	LHS&:=&\liminf_{\epsilon\to 0}\frac{1}{\epsilon}\int_{\RR^n}\abs{1-\frac{e^{Lg_\epsilon(x)}}{\int_{\RR^n}e^{Lg_\epsilon}\, {\rm d}\mu_c}}{\rm d}\mu_c(x) \\&\geq& \int_{\R^n}\liminf_{\epsilon\to 0}\frac{1}{\epsilon}\abs{1-\left(\frac{\epsilon+c}{c}\right)^\frac{n}{4}e^{\frac{c c_\epsilon}{2(1+c_\epsilon)}|x|^2}} \mathrm{d}\mu_c (x)\\&=&\frac{1}{4}\int_{\R^n}\left| -|x|^2+\frac{n}{c} \right| \mathrm{d}\mu_c (x),
\end{eqnarray*}
where the last term is positive and finite. Now,   dividing \eqref{contradiction-2-t} by $\epsilon>0$ and taking the limit as $\epsilon\to 0$, by the latter estimate, relation \eqref{first-contr}, and our assumption $\eta>1/2$, it follows that
$$0<\frac{1}{4}\int_{\R^n}\left| -|x|^2+\frac{n}{c} \right| \mathrm{d}\mu_c (x)\leq C_n\cdot\limsup_{\epsilon\to 0^+}\frac{ \left(\delta^{\sf Tal}_{\mu_c}(\nu_\epsilon)\right)^\eta}{\epsilon}=0,$$
a contradiction. 

\section{Hypercontractivity estimates for Hopf--Lax semigroups: proof of Theorem \ref{prop-HC-log-Sob}}\label{section-5}

\subsection{Proof of Theorem \ref{prop-HC-log-Sob}/(i)} 
	Let us fix a smooth function $u:\mathbb R^n\to \mathbb R$ verifying the growth condition \eqref{growth-u}, and for every $\tau>0$, consider  
	\begin{align*}
		F(\tau)=\bigg(\int_{\mathbb R^n} e^{q(\tau ){\bf Q}_\tau u }{\rm d}\mu_V \bigg)^{1/q(\tau )}=\Vert e^{{\bf Q}_\tau u}\Vert_{L^{q(\tau )}(\mathbb R^n,{\rm d}\mu_V)}.
	\end{align*}
	As mentioned above, relations \eqref{V-estimate} and \eqref{growth-u} imply that $F$ is well-defined; in particular, for every $\tau >0$, one has that
	$e^{\frac{q(\tau )}{2}{\bf Q}_\tau u}\in L^2(\mathbb R^n,{\rm d}\mu_V).$
	Furthermore,   a minor modification of the argument  from Bobkov and Ledoux \cite{BobkovLedoux}  (see also Balogh, Krist\'aly, and Tripaldi \cite[Proposition 4.1]{BKT}) shows that both $\tau \mapsto F(\tau )$ and $(\tau ,x)\mapsto {\bf Q}_\tau u(x)$ are locally Lipschitz on $(0,\infty)$ and $(0,\infty)\times \mathbb R^n$, respectively.    We know that for a.e.\ $(x,\tau )\in \mathbb R^n\times (0,\infty),$ the  Hamilton--Jacobi equation holds, i.e., 
	\begin{equation}\label{HJacobi}
		\frac{\partial}{\partial \tau}{\bf Q}_\tau u(x)+\frac{1}{2}|\nabla {\bf Q}_\tau u(x)|^2=0. 
	\end{equation}
	Besides $e^{\frac{q(\tau )}{2}{\bf Q}_\tau u}\in L^2(\mathbb R^n,{\rm d}\mu_V),$  we claim that $e^{\frac{q(\tau )}{2}{\bf Q}_\tau u}\in W^{1,2}(\mathbb R^n,{\rm d}\mu_V)$ for a.e.\ $\tau>0$. Indeed, due to \eqref{growth-u},    a simple estimate shows  that ${\bf Q}_\tau u\cdot e^{q(\tau ){\bf Q}_\tau u}\in L^1(\mathbb R^n,{\rm d}\mu_V)$. Moreover, since $\tau\mapsto F(\tau)^{q(\tau)}$ is locally Lipschitz on $(0,\infty)$, thus differentiable a.e., one has that $ -\infty < \frac{{\rm d}}{{\rm d}\tau}F(\tau)^{q(\tau)}<\infty$  for a.e.\ $\tau\in (0,\infty).$ Thus, by means of \eqref{HJacobi}, we have for a.e.\ $\tau\in (0,\infty) $ that
	\begin{eqnarray*}
		-\infty &<& \frac{{\rm d}}{{\rm d}\tau}F(\tau)^{q(\tau)} \\&=&\int_{\mathbb R^n}\frac{{\rm d}}{{\rm d}\tau} e^{q(\tau ){\bf Q}_\tau u }{\rm d}\mu_V\\&=&q'(\tau)\int_{\mathbb R^n} {\bf Q}_\tau u\cdot e^{q(\tau ){\bf Q}_\tau u} {\rm d}\mu_V- \frac{q(\tau)}{2}\int_{\mathbb R^n}|\nabla {\bf Q}_\tau u(x)|^2e^{q(\tau ){\bf Q}_\tau u} {\rm d}\mu_V. 
	\end{eqnarray*}
	The latter relation implies that $$\displaystyle\int_{\mathbb R^n} 
	\left|\nabla (e^{\frac{q(\tau )}{2}{\bf Q}_\tau u})\right|^2{\rm d}\mu_V<\infty\quad {\rm for\ a.e.}\ \tau>0,$$  concluding the proof of the claim. 
	
	A direct computation, combined with   \eqref{HJacobi}, shows that for a.e.\ $\tau >0$ one has 
	\begin{equation}\label{F-derivative}
		F'(\tau )=\frac{F(\tau )^{1-q(\tau )}}{q(\tau )^2}\left(q'(\tau ){\rm Ent}_{\mu_V}(e^{q(\tau ){\bf Q}_\tau u})-2\int_{\mathbb R^n} 
		\left|\nabla (e^{\frac{q(\tau )}{2}{\bf Q}_\tau u})\right|^2{\rm d}\mu_V\right).
	\end{equation}
	By the convexity of $s\mapsto s\log s$ on $(0,\infty)$ and Jensen's inequality, we have ${\rm Ent}_{\mu_V}(e^{q(\tau ){\bf Q}_\tau u})\geq 0$. Therefore, in view of  $q'\leq c$ and the definition of the Bakry--\'Emery log-Sobolev deficit, relation \eqref{F-derivative} implies for a.e.\ $\tau>0$ that
	\begin{eqnarray}\label{F-estimate-c}
		F'(\tau )&\leq& \frac{F(\tau )^{1-q(\tau )}}{q(\tau )^2}\left(c\cdot {\rm Ent}_{\mu_V}(e^{q(\tau ){\bf Q}_\tau u})-2\int_{\mathbb R^n} 
		\left|\nabla (e^{\frac{q(\tau )}{2}{\bf Q}_\tau u})\right|^2{\rm d}\mu_V\right)\\&=& -c\frac{F(\tau )}{q(\tau )^2}\delta^{\sf LSI}_{\mu_V}(e^{\frac{q(\tau )}{2}{\bf Q}_\tau u}). \nonumber
	\end{eqnarray}
	If $t>0$ is fixed arbitrarily, an integration on $(0,t)$ yields relation \eqref{HC-LSI-deficits}, i.e.,  
	$$\delta^{\sf HC}_t(u)=\log \frac{F(0)}{F(t)}\geq c\int_0^t \frac{1}{q^2(\tau)}\delta^{\sf LSI}_{\mu_V}(e^{\frac{q(\tau)}{2}{\bf Q}_\tau u}) {\rm d}\tau.$$

\subsection{Proof of Theorem \ref{prop-HC-log-Sob}/(ii)}  	Assume now that equality holds in \eqref{hypercontractivity-estimate-Gauss-1} for some $t >0$ and a smooth function $u:\mathbb R^n\to \mathbb R$ verifying the growth condition \eqref{growth-u}. In particular, $\delta^{\sf HC}_t(u)=0$, and by taking \eqref{HC-LSI-deficits} into account, one has that \begin{equation}\label{def-lsi-zero}
		\delta^{\sf LSI}_{\mu_V}(e^{\frac{q(\tau)}{2}{\bf Q}_\tau u})=0\ \ \ {\rm for\ a.e.}\   \tau\in (0,t). 
	\end{equation}
%
%
		Applying  the inequality \eqref{relation-deficits} for the choice $f_0:=e^{\frac{q(\tau)}{2}{\bf Q}_\tau u}/F(\tau)^{\frac{q(\tau)}{2}}$, due to \eqref{def-lsi-zero}, it follows that 
	\begin{equation}\label{def-tal-zero}
		\delta^{\sf Tal}_{\mu_V}(e^{q(\tau){\bf Q}_\tau u}/F(\tau)^{{q(\tau)}}{\rm d}\mu_V)=0\ \ \ {\rm for\ a.e.}\   \tau\in (0,t). 
	\end{equation}
	By the characterization of the equality case in the Talagrand transport inequality (see Theorem \ref{Theorem-Talagrand}/(ii)), it turns out that there exists $x_0^\tau\in {\rm Ker}(\nabla^2 V -cI_n)$ such that 
	$$e^{q(\tau){\bf Q}_\tau u(x)}/F(\tau)^{q(\tau)}=e^{V(x)-V(x +x_0^\tau)}\ \ \ {\rm for\ a.e.}\   (x,\tau)\in \mathbb R^n\times (0,t).$$
	Since $V(x+x_0^\tau)- V(x)=c\langle x_0^\tau,x\rangle  + V(x_0^\tau)-V(0)$, cf. Proposition \ref{V-properties},  it follows by the previous relation that 
	$${\bf Q}_\tau u(x)=\langle \overline x_0^\tau,x\rangle+c_\tau\ \ \ {\rm for\ a.e.}\   (x,\tau)\in \mathbb R^n\times (0,t),$$
	for some $c_\tau\in \mathbb R$ and $\overline x_0^\tau\in  {\rm Ker}(\nabla^2 V -cI_n)$; in fact, $\overline x_0^\tau=-\frac{c}{q(\tau)}x_0^\tau.$ By continuity reasons, the latter relation is valid for every $x\in \mathbb R^n$ and a.e.\ $\tau\in (0,t).$
	Moreover,   $c_\tau={\bf Q}_\tau u(0)$ for a.e.\ $\tau\in (0,t)$. Consider $A\subset (0,t)$ having full measure where the latter relation holds, i.e., 
	\begin{equation}\label{Q-affine}
		{\bf Q}_\tau u(x)=\langle \overline x_0^\tau,x\rangle+{\bf Q}_\tau u(0)\ \ \ {\rm for\ every}\   (x,\tau)\in \mathbb R^n\times A.
	\end{equation}
	By the semigroup structure of ${\bf Q}_\tau$,    for every $\tau, \eta\in (0,t)$ and every $x \in  \mathbb R^n,$ we have the relation 
	\begin{equation} \label{semigroup-struc} 
		{\bf Q}_{\tau-\eta}{\bf Q}_\eta u(x) = {\bf Q}_\tau u(x).
	\end{equation} Moreover, note that  for almost every pair $(\tau,\eta)\in A\times A$ with $\tau>\eta$, one has that $\tau-\eta\in A$. For such a pair, the definition of the Hopf--Lax formula  (see \eqref{inf-convolution-0})  and relation  \eqref{Q-affine} imply that
	$${\bf Q}_{\tau-\eta}{\bf Q}_\eta u(x)={\bf Q}_{\tau-\eta}(\langle \overline x_0^\eta,\cdot \rangle+{\bf Q}_\eta u(0))(x)=\langle \overline x_0^\eta,x \rangle+ {\bf Q}_\eta u(0) - |\overline x_0^\eta|^2 \frac{\tau-\eta}{2}.$$
	Thus, an identification based on 
	\eqref{semigroup-struc} yields 
	$$\langle \overline x_0^\eta,x \rangle+ {\bf Q}_\eta u(0) - |\overline x_0^\eta|^2 \frac{\tau-\eta}{2}=\langle \overline x_0^\tau,x\rangle+{\bf Q}_\tau u(0),\ \ {x\in \mathbb R^n}.$$
	In particular, it follows that for a full measure  set  $A_0\subset (0,t)$, one has that $\overline x_0^\tau=x_0$ for every $\tau\in A_0;$  clearly,  $x_0\in {\rm Ker}(\nabla^2 V -cI_n)$.  Thus, in view of  relation  \eqref{Q-affine}, it follows that 
	$${\bf Q}_\tau u(x)=\langle  x_0,x\rangle+{\bf Q}_\tau u(0)\ \ \ {\rm for\ every}\   (x,\tau)\in \mathbb R^n\times A_0.$$ 
	Now, by continuity arguments, it follows that the last equality is in fact valid on $\mathbb R^n\times (0,t).$ Using \eqref{growth-u}, by letting $\tau\to 0^+$ in the previous relation, it follows that $u(x)=\langle  x_0,x\rangle+  u(0)$ for every $x\in \mathbb R^n.$
	
	Furthermore, equality in \eqref{hypercontractivity-estimate-Gauss-1} implies that there should be equality in \eqref{F-estimate-c} for a.e.\ $\tau\in (0,t)$; namely, we have 
	$$\left(q'(\tau )-c\right){\rm Ent}_{\mu_V}(e^{q(\tau ){\bf Q}_\tau u})=0\ \ \ {\rm for\ a.e.}\   \tau\in (0,t).$$ 
	
	Assume that there exists $\tau\in (0,t)$ such that ${\rm Ent}_{\mu_V}(e^{q(\tau ){\bf Q}_\tau u})=0$. By the strict convexity of $s\mapsto s\log s$ on $(0,\infty)$, it follows that for a.e.\ $x\in \mathbb R^n$  we have $e^{q(\tau ){\bf Q}_\tau u(x)}=C$ for some $C>0.$ In particular, this implies that $x_0=0$, thus $u={\bf Q}_\tau u$ is a constant; in this case, $q$ can be arbitrary. 
	If ${\rm Ent}_{\mu_V}(e^{q(\tau ){\bf Q}_\tau u})\neq 0$ for every $\tau\in (0,t)$, then $x_0\neq 0$ and $q'(\tau )-c=0$, hence $q(\tau)=c\tau +a$ for every $\tau\in (0,t)$ with $a=q(0)>0$.
	
	Conversely, if $u(x)=\langle  x_0,x\rangle+  c_0$, where $x_0\in {\rm Ker}(\nabla^2 V -cI_n)$ and $q(\tau)=a+c\tau$ for $\tau>0$, it follows that ${\bf Q}_\tau u(x)=\langle  x_0,x\rangle+c_0-\frac{\tau}{2}|x_0|^2.$ According to Proposition \ref{V-properties}/(ii), one has
	$$\int_{\mathbb R^n} e^{q(\tau ){\bf Q}_\tau u }{\rm d}\mu_V=e^{q(\tau)(c_0-\frac{\tau}{2}|x_0|^2)+c|\overline x_\tau|^2-V(\overline x_\tau)+V(0)}\ \  {\rm and}\ \ \int_{\mathbb R^n} e^{a u }{\rm d}\mu_V=e^{ac_0+\frac{a^2}{c}|x_0|^2-V(\frac{a}{c}x_0)+V(0)},$$
	where $\overline x_\tau=\frac{q(\tau) }{c}x_0.$ Thus, due to  Proposition \ref{V-properties}/(iii), the 
	hypercontractivity deficit is 
	\begin{eqnarray*}
		\delta^{\sf HC}_\tau(u)&=&\log \frac{\|e^{u}\|_{L^{a}(\mathbb R^n,{\rm d}\mu_V)}}{\|e^{{\bf Q}_\tau u}\|_{L^{q(\tau )}(\mathbb R^n,{\rm d}\mu_V)}}\\&=&\left(\frac{a}{c}+\frac{\tau}{2}\right)|x_0|^2-\frac{1}{q(\tau)}\left(c|\overline x_\tau|^2-V(\overline x_\tau)+V(0)\right)-\frac{1}{a}\left(V(\frac{a}{c}x_0)-V(0)\right)\\&=&0,
	\end{eqnarray*}
	which concludes the proof. \\
	
\noindent 	\textbf{Acknowledgments}. The authors thank Zolt\'an M. Balogh,  Emanuel Indrei, and Michel Ledoux   for
	stimulating conversations in the early stages of the manuscript. 

\vspace{1cm}

\noindent \textbf{\textit{Declarations}}\\

\noindent \textbf{Data Availability}. Data sharing is not applicable to this article as no datasets were
generated or analyzed during the current study.\\

\noindent \textbf{Conflict of interest}. The authors state that there is no conflict of interest.

%
%
%
%

\end{document}